\documentclass[11pt,reqno]{amsart}

\usepackage{docmute}

\usepackage{graphicx}
\usepackage[margin=25truemm]{geometry}
\usepackage{amsmath,amssymb,amsthm,amsfonts}
\usepackage{hyperref}
\usepackage{tikz-cd}
\usepackage{setspace}
\usepackage{DotArrow}
\usepackage{microtype}
\usepackage{stix2}

\numberwithin{equation}{subsection}
\usetikzlibrary{calc}

\hypersetup{% hyperrefオプションリスト
setpagesize=false,
 bookmarksnumbered=true,%
 bookmarksopen=true,%
 colorlinks=true,%
 linkcolor=red,
 citecolor=blue,
}

\let\it\relax

\newcommand{\it}[1]{\textit{#1}}

\newcommand{\op}{\mathrm{op}}
\newcommand{\id}{\mathrm{id}}

\newcommand{\dgcat}{\mathrm{dgcat}}

\newcommand{\dg}{\mathrm{dg}}
\newcommand{\sset}{\mathrm{sSet}}

\newcommand{\SCG}{\mathrm{SCG}}
\newcommand{\htpytt}{\mathcal{H}_{\mathrm{3t}}}

\newcommand{\MA}{\mathcal{A}}

\newcommand{\MC}{\mathcal{C}}

\newcommand{\MS}{\mathcal{S}}
\newcommand{\MT}{\mathcal{T}}

\newcommand{\MD}{\mathcal{D}}

\newcommand{\RC}{\mathrm{C}}

\newcommand{\RD}{\mathcal{D}}

\newcommand{\tauleq}[1]{\tau_{\leq #1}}
\newcommand{\taugeq}[1]{\tau_{\geq #1}}

\DeclareMathOperator{\rep}{\mathrm{rep}}
\DeclareMathOperator{\Mod}{\mathrm{Mod}}

\DeclareMathOperator{\Ker}{\mathrm{Ker}}
\DeclareMathOperator{\Cok}{\mathrm{Cok}}

\DeclareMathOperator{\Cone}{\mathrm{Cone}}
\DeclareMathOperator{\Cocone}{\mathrm{Cocone}}

\theoremstyle{definition}
\newtheorem{ttheorem}{Theorem}
\newtheorem{theorem}{Theorem}[section]
\newtheorem{proposition}[theorem]{Proposition}
\newtheorem*{proposition*}{Proposition}

\newtheorem{lemma}[theorem]{Lemma}
\newtheorem{cor}[theorem]{Corollary}
\newtheorem*{cor*}{Corollary}

\newtheorem{definition}[theorem]{Definition}
\newtheorem*{definition*}{Definition}

\newtheorem{remark}[theorem]{Remark}
\newtheorem{example}[theorem]{Example}

\begin{document}

\title[Abelian $n$-truncated DG-categories]{Higher-dimensional generalization of abelian categories\\ via DG-categories}

\author{Nao Mochizuki}

\address{Graduate School of Mathematics, Nagoya University, Furocho, Chikusaku, Nagoya 464-8602, Japan}
\email{mochizuki.nao.n8@s.mail.nagoya-u.ac.jp}

\begin{abstract}
In this paper, we introduce abelian $n$-truncated DG-categories as an $n$-dimensional analogue of abelian categories in the setting of DG-categories. When $n=1$, this recovers ordinary abelian categories, and when $n=\infty$, it corresponds to stable DG-categories \cite{chen2024exact1}. This notion serves as a DG-categorical 
analogue of abelian $(n,1)$-categories introduced in the context of $(n,1)$-categories by \cite{stefanich2023derived}. We show that the homotopy categories of abelian $n$-truncated DG-categories acquire the structure of extriangulated and pretriangulated categories. Furthermore, we develop a general theory of abelian $n$-truncated DG-categories, including the analogues of the existence epi-mono factorizations of morphisms, as in classical abelian categories.
\end{abstract}

\maketitle

\tableofcontents

\setcounter{section}{-1}

\section{Introduction}

Abstract settings for homological algebra, such as triangulated categories, abelian categories, and exact categories, share the notion of a "short exact sequence" or its analogues. 
Nakaoka and Palu \cite{nakaoka2019extriangulated} generalized these categories by introducing the notion of extriangulated categories, thereby providing a unified framework to understand their interrelations and enabling a common language for discussion. 
Furthermore, Xiaofa Chen \cite{chen2023exact} introduced the concept of exact DG-categories—DG-categories equipped with additional structures—aimed at constructing DG-enhancements of extriangulated categories. 
In particular, it was shown that the homotopy category of an exact DG-category naturally admits the structure of an extriangulated category, generalizing the result that a pretriangulated DG-category yields a DG-enhancement of a triangulated category (cf.\ \cite{bondal1991enhanced}).

When abelian categories and triangulated categories are viewed as extriangulated categories, they exhibit notable properties that are not generally present in exact categories. For instance, every morphism admits kernel-like and cokernel-like notions and can be decomposed into deflation and inflation. This naturally leads to the question: how can extriangulated categories, which generalize shared properties of abelian and triangulated categories, be axiomatized?

Regarding the second motivation, within the context of the representation theory of algebras, \cite{adachi2014tilting} developed a framework for classifying torsion classes of functorially finite modules using support $\tau$-tilting modules and 2-term silting complexes. As a natural generalization of 2-term silting theory, the notion of $(n+1)$-term silting objects has been studied (cf.\ \cite{gupta2024d}), and results analogous to those of \cite{adachi2014tilting} have been established in \cite{zhou2024tilting}. However, \cite{zhou2024tilting} demonstrates that the objects corresponding to $(n+1)$-term silting complexes reside in $n$-extended module categories (cf.\ \cite[Theorem 4.7]{zhou2024tilting}). Accordingly, it becomes essential to develop a foundational theory for $n$-extended modules and, more generally, for the $n$-extended hearts associated with $t$-structures.

In \cite{gupta2024d} and \cite{zhou2024tilting}, $n$-extended module categories are treated as extension-closed subcategories of triangulated categories, necessitating frequent reference to the surrounding triangulated structure. By providing an axiomatic framework for $n$-extended module categories, a more concise and self-contained understanding of the aforementioned theories can be achieved.

As a third motivation, \cite{nakaoka2008cohomology} and \cite{dupont2008abelian} introduced a $(2,1)$-categorical analogue of abelian categories. This concept arose as an axiomatization of the $2$-category formed by $2$-dimensional analogues of abelian groups, known as symmetric categorical groups \cite{vitale2002picard}. Furthermore, Stefanich introduced abelian $(n,1)$-categories \cite{stefanich2023derived} in the context of $(\infty,1)$-categories. In this paper, we develop analogues of these concepts in the setting of DG-categories.

\subsection*{Structure and main results of this paper}

In \S\ref{sec: preliminalies}, we discuss the theory of exact DG-categories, which plays a crucial role in this study. We introduce lemmas concerning homotopy short exact sequences in general additive DG-categories and provide the definition of exact DG-categories. This section is primarily based on the results of \cite{chen2023exact}.

In \S\ref{sec: preabel}, before addressing abelian $n$-truncated DG-categories, we consider the homotopy category of preabelian DG-categories. A preabelian DG-category is an additive DG-category that admits homotopy kernels and homotopy cokernels. When an additive category is regarded as a DG-category, it naturally corresponds to a preabelian category in the usual sense. The main theorem of this section is as follows:

\begin{proposition*}[Proposition \ref{prop: left triang}]
Let $\MA$ be a finitely complete DG-category. Then, $(H^0(\MA), \Omega, \Delta)$ forms a left triangulated category (cf. Definition \ref{def: left triangulated cat}).
\end{proposition*}

\begin{ttheorem}[Theorem \ref{thm: pretriang}]
Let $\MA$ be a preabelian DG-category. Then, $H^0(\MA)$ naturally admits a pretriangulated structure $(\Omega, \Sigma, \Delta, \nabla)$.
\end{ttheorem}

\noindent
This pretriangulated structure is not defined by weakening the octahedral axiom in the sense of triangulated categories but is instead the framework introduced by \cite{beligiannis2007homological}. Pretriangulated categories, introduced by Beligiannis and Reiten \cite{beligiannis2007homological}, generalize both preabelian and triangulated categories, providing a unified framework to handle the theories of such torsion pairs. Definitions and details of pretriangulated categories are presented in Appendix \ref{app: B}.

In \S\ref{sec: abel}, we introduce the concept of $n$-monomorphisms (cf. Definition \ref{def: n-mono}) as an analog of monomorphisms in abelian categories, and using this concept, we define abelian $n$-truncated DG-categories.

\begin{definition*}[Definition \ref{def: abel def}]
Let $\MA$ be an additive $n$-truncated DG-category. Then, $\MA$ is called an \it{abelian $n$-truncated DG-category} if it satisfies the followings:
\begin{itemize}
\item [(i)] $\MA$ is a preabelian $n$-truncated DG-category.
\item [(ii)] For any right exact $3$-term homotopy  complex:
$$
\begin{tikzcd}
 X \ar[r, "f", swap] \ar[rr, "0", bend left=45, "{}" name=zero]
 & |[alias=Y]| Y \ar[r, "g", swap] & Z
\ar[Rightarrow, from=Y, to=zero, "h", shorten >=1mm]
\end{tikzcd}
$$
if $f$ is an $n$-monomorphism, then this $3$-term homotopy complex is a homotopy short exact sequence.
\item [(iii)] The dual of (ii).
\end{itemize}
\end{definition*}

We show that abelian $n$-truncated DG-categories naturally become exact DG-categories, and the natural extriangulated structure in their homotopy category is compatible with the pretriangulated structure discussed in \S\ref{sec: preabel}.

\begin{proposition*}[Proposition \ref{prop: the homotopy category of abelian}]
Let $\MA$ be an abelian $n$-truncated DG-category. By regarding every homotopy short exact sequence as a conflation, $\MA$ becomes an exact DG-category. That is, inflations in $\MA$ are equivalent to $n$-monomorphisms, and deflations are equivalent to $n$-epimorphisms.
\end{proposition*}

\noindent
The terms such as conflation are defined in Definition \ref{definition: exact dg}.

\begin{cor*}[Corollary \ref{cor: long exact induced abelian}]
Let $\MA$ be an abelian $n$-truncated DG-category. Let $H^0(\MA)$ be denoted by $\MC$, and $(\MC, \mathbb{E}, \mathfrak{s})$ the natural extriangulated structure derived from the exact DG-category. Let 
\begin{tikzcd}[column sep = 15]
{A}\ar[r,"f"] & {B}\ar[r,"g"] & {C}\ar[r,dashed,"\delta"] & {}
\end{tikzcd}
be a conflation in this extriangulated category. Then, the following long exact sequence is induced:
$$
\begin{tikzcd}[row sep=10,column sep=40]
0 \ar[r] & \MC(-, \Omega^{n-1} A) \ar[r, "\Omega^{n-1} f \circ -"] & \MC(-, \Omega^{n-1} B) \ar[r, "\Omega^{n-1} g \circ -"] & \MC(-, \Omega^{n-1} C) \ar[r] & \cdots \\
{} \ar[r] & \MC(-, \Omega A) \ar[r, "\Omega f \circ -"] & \MC(-, \Omega B) \ar[r, "\Omega g \circ -"] & \MC(-, \Omega C) \ar[r] & {} \\
{} \ar[r] & \MC(-, A) \ar[r, "f \circ -"] & \MC(-, B) \ar[r, "g \circ -"] & \MC(-, C) \ar[r, "\delta \cdot -"] & {} \\
{} \ar[r] & \mathbb{E}(-, A) \ar[r] & \mathbb{E}(-, B) \ar[r] & \mathbb{E}(-, C) \ar[r] & \cdots \\
\cdots \ar[r] & \mathbb{E}^m(-, A) \ar[r] & \mathbb{E}^m(-, B) \ar[r] & \mathbb{E}^m(-, C) \ar[r] & \cdots \\
\end{tikzcd}
$$
\end{cor*}

\noindent
We also describe the factorization of morphisms, which is a prominent property of abelian $n$-truncated DG-categories. In abelian categories, every morphism can be decomposed into an epimorphism and a monomorphism. Similarly, in abelian $n$-truncated DG-categories, such factorization can be described as follows:

\begin{ttheorem}[Theorem \ref{thm: the factorization theorem}]
Let $\MA$ be an abelian $n$-truncated DG-category, and let $f$ be a closed morphism of degree $0$ in $\MA$. Then, the factorization of $f$ into a 1-epimorphism and an $n$-monomorphism exists uniquely up to homotopy equivalence.
\end{ttheorem}

\noindent
Dually, a factorization into an $n$-epimorphism and a 1-monomorphism is also obtained. In abelian categories, these two factorizations coincide, which can also be understood in this context. Much of the proof of this theorem is inspired by \cite{nakaoka2008cohomology}.

 We demonstrate that the DG enhancement of the $n$-extended heart of an algebraic triangulated category is an abelian $n$-truncated DG-category:

\begin{ttheorem}[Theorem \ref{thm: induced by t-structure}]
Let $\MD_{\dg}$ be a non-positive stable DG-category, and let $(\MD_{\dg}^{\leq 0}, \MD_{\dg}^{\geq 0})$ be its $t$-structure. Then, the $n$-extended heart $\MD_{\dg}^{[-n+1,0]}$ forms an abelian $n$-truncated DG-category.
\end{ttheorem}

Without using the language of DG-categories, theorems so far are stated as follows:

\begin{ttheorem}[Theorem \ref{thm: h0}]
Let $\MD$ be a algebraic triangulated categories and $(\MD^{\leq0},\MD^{\geq 0})$ be a $t$-structure. Then 
the following holds. 
\begin{itemize}
\item[(i)] $\MD^{[-n+1,0]}$ admits a pretriangulated structure $(\Omega,\Sigma,\Delta,\nabla)$ where:
\begin{itemize}
\item $\Omega:=\tauleq{0}\circ[-1]$ and $\Sigma:=\taugeq{-n+1}\circ [1]$
\item $\Delta$ is the class of triangles isomorphic to the following triangles:
$$
\begin{tikzcd}
\tauleq{0}(Y[-1])\ar[r]&\tauleq{0}(\Cone f[-1]) \ar[r, swap]  & |[alias=Y]| X \ar[r, "f"', swap] & Y 
\end{tikzcd}
$$
where $f\colon X\rightarrow Y$ is a morphism of $\MD^{[-n+1,0]}$.
\item $\nabla$ is the class of triangles isomorphic to the following triangles:
$$
\begin{tikzcd}
X\ar[r,"f"]&Y\ar[r]&\taugeq{-n+1}\Cone f\ar[r]&\taugeq{-n+1}(X[1])
\end{tikzcd}$$
\end{itemize}
\item [(ii)] $\MD^{[-n+1,0]}$ admits a extriangulated structure suth that:
\begin{itemize}
\item Conflations are admissible triangles \begin{tikzcd}[column sep = 15]
{A}\ar[r,"f"] & {B}\ar[r,"g"] & {C}\ar[r,dashed,"\delta"]&{}
\end{tikzcd} in $\MD$ which every terms are in $\MD^{[-n+1,0]}$
\item For a morphism $f\colon X\rightarrow Y$ in $\MD^{[-n+1,0]}$, The followings are equivalent:
\begin{itemize}
\item [(1)] $f$ is a inflation.
\item [(2)] $H_{\MD}^{-n+1}(f)$ is monomorphic in the heart of fixed $t$-structure.
\item [(3)] $\Omega^{n-1}(f)$ is monomorphic in $\MD^{[-n+1,0]}$.
\end{itemize}
\item For a morphism $f\colon X\rightarrow Y$ in $\MD^{[-n+1,0]}$, The followings are equivalent:
\begin{itemize}
\item [(1)] $f$ is a deflation.
\item [(2)] $H_{\MD}^{0}(f)$ is epimorphic in the heart of fixed $t$-structure.
\item [(3)] $\Sigma^{n-1}(f)$ is epimorphic in $\MD^{[-n+1,0]}$.
\end{itemize}
\end{itemize}
\item [(iii)] For a morphism $f\colon X\rightarrow Y$ in $\MD^{[-n+1,0]}$, there exists a factorization $f=m_n\circ e_1$ suth that:
\begin{itemize}
\item $H^{-n+1}_\MD(m_n)$ is a monomorphism in the heart of fixed $t$-structure, or equivalently, $\Omega^{n-1}(m_n)$ is a monomorphism in $\MD^{[-n+1,0]}$.
\item $H^{-n+1}_\MD (e_1)$ is a epimorphism and $H^{\geq -n+1}_\MD(e_1)$ is isomorphism in the heart of fixed $t$-structure, or equivalently, $e_1$ is a epimoprhism and $\Sigma^{\geq 1}(e_1)$ is a isomorphism in $\MD^{[-n+1,0]}$.
\end{itemize}
 and this factorization is unique up to isomorphisms.
\item [(iii)'] The dual statements of (iii)
\end{itemize}
\end{ttheorem}

Finally, in \S \ref{sec :5}, we discuss the connection between abelian $(2,1)$-categories and relatively exact $2$-categories introduced in \cite{nakaoka2008cohomology}. We establish the following result:

\begin{ttheorem}[Theorem \ref{thm: abelian 2-truncated}]
Let $\MA$ be a $2$-truncated DG-category. Then $\MA$ is an abelian $2$-truncated DG-category if and only if $\MA_\SCG$ is a relatively exact $2$-category without $(\text{a3-1})$.
\end{ttheorem}

Moreover, the following assertion also holds in the setting of $(n,1)$-categories.

\begin{proposition*}[Proposition \ref{thm: stef}]
Let $\MA$ be a non-positive DG-categories. Then $\MA$ is an abelian $n$-truncated DG-category if and only if $N_{\dg}(\MA)$ is an abelian $(n,1)$-category in the sense of \cite{stefanich2023derived}.
\end{proposition*}

% 6/27　ここまでやった．

\section{Preliminaries}\label{sec: preliminalies}

In this section, we introduce homotopy diagrams and exact DG categories that will be used in subsequent discussions. This chapter primarily follows \cite{chen2023exact}. 

From now on, let $\mathcal{A}$ be a DG-category, which we always assume to be non-positive. 
For an object $A \in \mathcal{A}$, we denote the representable DG-functor $\mathcal{A}(-,A)$ by $A^\land$ and $\mathcal{A}(A,-)$ by $A^\lor$.

To enhance readability, we adopt the following notation.

\begin{definition}
Let $f_{ij} \colon A_i \to A_j \ (0 \leq i < j \leq 2)$ be closed morphisms in $\MA$ such that $|f_{02}| = |f_{01}| + |f_{12}|$. If there exists a morphism $h \colon A_0 \to A_2$ in $\MA$ of degree $|f_{02}| - 1$ satisfying $d(h) = f_{02} - f_{12} \circ f_{01}$, we represent it as follows:
\[
\begin{tikzcd}
A_0 \ar[rr, "f_{02}" name=f_02] \ar[rd, "f_{01}", swap] & & A_2 \\
& |[alias=a_1]|A_1 \ar[ru, "f_{12}", swap] & 
\ar[Rightarrow, from=a_1, to=f_02, "h", shorten >=2mm, shorten <=2mm]
\end{tikzcd}
\]
\end{definition}

\begin{definition}[{\cite[Definition 3.39]{chen2023exact}}]
A diagram of the following form, where $|f| = |g| = 0$ and $|h| = -1$, is called a \it{$3$-term homotopy complex} and is denoted simply by $(f, g, h)$:
\[
\begin{tikzcd}
X \ar[r, "f", swap] \ar[rr, "0", bend left=45, "{}" name=zero] & |[alias=Y]| Y \ar[r, "g", swap] & Z
\ar[Rightarrow, from=Y, to=zero, "h", shorten >=1mm]
\end{tikzcd}
\]
\end{definition}

\begin{definition}[{\cite[Definition 3.39]{chen2023exact}}]
Let $(f, g, h)$ and $(f', g', h')$ be $3$-term homotopy complexes. Consider a sextuple $(x, y, z, s, s', t)$ of morphisms.  ($t \colon X \to Z'$ is not explicitly shown in the diagram.)  
\[
\begin{tikzcd}
X \ar[r, "f"] \ar[d, "x"] \ar[rr, bend left=45, "0" name=0] & |[alias=B_0]| Y \ar[r, "g"] \ar[d, "y"] & |[alias=C_0]| Z \ar[d, "z"] \\
|[alias=A]| X' \ar[r, "f'", swap] \ar[rr, bend right=45, "0" name=0', swap] & |[alias=B_1]| Y' \ar[r, "g'", swap] & Z'
\ar[Rightarrow, from=C_0, to=B_1, "s'"]
\ar[Rightarrow, from=B_0, to=A, "s"]
\ar[Rightarrow, from=B_0, to=0, "h", shorten >=1mm]
\ar[Rightarrow, from=B_1, to=0', "h'", shorten >=1mm]
\end{tikzcd}
\]
The sextuple $(x, y, z, s, s', t)$ is called a morphism from $(f, g, h)$ to $(f', g', h')$ if the following conditions are satisfied:
\begin{itemize}
    \item $|x| = |y| = |z| = 0$, $|s| = |s'| = -1$, $|t| = -2$.
    \item $d(x) = 0$, $d(y) = 0$, $d(z) = 0$, $d(s) = f' \circ x - y \circ f$, $d(s') = g' \circ y - z \circ g$\\
    $d(t) = z \circ h - s' \circ f - g' \circ s - h' \circ x$.
\end{itemize}
\end{definition}

\begin{definition}[{\cite[Definition 3.39]{chen2023exact}}]
The category $\htpytt(\MA)$ is defined as follows:
\begin{itemize}
    \item The objects are $3$-term homotopy complexes in $\MA$.
    \item The morphisms are given by the morphisms between $3$-term homotopy complexes, as defined above, modulo an equivalence relation $\sim$. For details on the equivalence relation $\sim$, see \cite[Definition 3.39]{chen2023exact}.
\end{itemize}
\end{definition}

\begin{definition}[{\cite[Definition 3.49]{chen2023exact}}]
Let $f \colon A \to A'$ be a closed morphism of degree $0$ in $\MA$. Consider the following $3$-term homotopy complex:
\begin{equation}\label{kernel}
\begin{tikzcd}
\Ker f \ar[r, "k_f", swap] \ar[rr, "0", bend left=45, "{}" name=zero] & |[alias=Y]| A \ar[r, "f", swap] & A' \quad
\ar[Rightarrow, from=Y, to=zero, "h_f", shorten >=1mm]
\end{tikzcd}
\end{equation}
This $3$-term homotopy complex induces the morphism in $\RD_{\dg}(\MA)$:
\[
\begin{bmatrix} k_f^\land \\ h_f^\land \end{bmatrix} \colon (\Ker f)^\land \to \Cone(f^\land)[-1].
\]
If this morphism induces an isomorphism on the $i$-th cohomology for all $i \leq 0$, the $3$-term homotopy complex is called the \it{homotopy kernel} of $f$. Alternatively, $(k_f, f, h_f)$ is said to be \it{left exact}.
\end{definition}

\begin{remark}
The homotopy kernel of $f$, if it exists, is unique up to homotopy equivalence.
\end{remark}

The homotopy cokernel is defined as the homotopy kernel in the opposite DG-category $\MA^\op$.

To describe the universality of the homotopy kernel, we introduce the concept of a filler. This notion serves as a substitute for the "commutativity" of diagrams. Let $x \colon X \to A$ be a closed morphism of degree $i$. Consider the following diagram:
\begin{equation}\label{dia: ker}
\begin{tikzcd}
K \ar[r, "k"] \ar[rr, "0" name=zero, bend left=45] & |[alias=Y]| A \ar[r, "f"] & A' \\
X \ar[ru, "x", swap] \ar[rru, bend right=30, "0" name=zero', swap] & &
\ar[Rightarrow, from=Y, to=zero, "h_f", shorten >=1mm]
\ar[Rightarrow, from=Y, to=zero', "h_x", swap, shorten <=1mm, shorten >=2mm]
\end{tikzcd}
\end{equation}

\begin{definition}
Given the diagram \eqref{dia: ker} and a closed morphism $x' \colon X \to K$ of degree $i$, a pair of morphisms $\overline{x'} \colon X \to A$ and $\widetilde{x'} \colon X \to A'$ is called a \emph{filler} for $x'$ in this diagram if the following conditions are satisfied:
\begin{itemize}
    \item [(i)] $|\overline{x'}| = i-1$, $|\widetilde{x'}| = i-2$,
    \item [(ii)] $d(\overline{x'}) = x - k \circ x'$,
    \item [(iii)] $d(\widetilde{x'}) = h_f \circ x' - f \circ \overline{x'} - h_x$.
\end{itemize}
 Alternatively, we say that $x'$ \emph{admits a filler} $(\overline{x'}, \widetilde{x'})$.
% The diagram is represented as follows:
% $$
% \begin{tikzcd}
% K \ar[r,"k"] \ar[rr,"0", bend left=45, "" name=zero] & |[alias=Y]| A \ar[r,"f"] & A'\\
% X \ar[ru,"x'" swap] \ar[rru,bend right=30,"0" name=zero',swap] & &
% \ar[Leftarrow, from=zero, to=Y, "h_f", shorten <=1mm]
% \ar[Leftarrow, from=zero', to=Y, "h_x" swap, shorten >=1mm, shorten <=2mm]
% \end{tikzcd}
% $$
\end{definition}

The following proposition states the universality of the homotopy kernel.

\begin{proposition}[{\cite[Lemma 3.50]{chen2023exact}}]
\label{prop: univ of kernels}
Consider the following $3$-term homotopy complex:
\begin{equation}\label{dia: left ex}
\begin{tikzcd}
K \ar[r, "k"] \ar[rr, "0" name=zero, bend left=45] & |[alias=Y]| A \ar[r, "f"] & A'
\ar[Rightarrow, from=Y, to=zero, "h_f", shorten >=1mm]
\end{tikzcd}
\end{equation}
This is a homotopy kernel of $f$ if and only if the following conditions are satisfied:
\begin{itemize}
    \item [(i)] Let $x \colon X \to A$ be a degree $i$ closed morphism, and let $h_x \colon X \to A'$ be a degree $i-1$ morphism satisfying $d(h_x) = -f \circ x$. Then, there exists a closed morphism $x' \colon X \to K$ of degree $i$ that admits a filler.
    \item [(ii)] In (i), if both $x' \colon X \to K$ and $x'' \colon X \to K$ admits a fillers, there exists a morphism $h \colon X \to K$ of degree $i-1$ such that $d(h) = x' - x''$.
\end{itemize}
\end{proposition}

\begin{remark}
In Proposition \ref{prop: univ of kernels}, the pair $(x, h_x)$ determines an element of degree $i$ in $(\Cone(f^\land)[-1])(X)$, and the corresponding element in $K^\land(X)$ is $x'$.
\end{remark}

The opposite category of a DG-category arises signs in compositions. Thus, when defining the dual concepts, attention must be paid to these signs.

\begin{definition}
Consider the following diagram and a closed morphism $x' \colon C \to X$ of degree $i$. If a pair of morphisms $\overline{x'} \colon A' \to X$ and $\widetilde{x'} \colon A \to X$ satisfies the following equations, the pair is called a \it{cofiller} for $x'$ in the diagram, or we say $x'$ \it{admits a cofiller} $(\overline{x'}, \widetilde{x'})$:
\begin{itemize}
    \item [(i)] $|\overline{x'}| = i-1$, $|\widetilde{x'}| = i-2$.
    \item [(ii)] $d(\overline{x'}) = x - x' \circ c$.
    \item [(iii)] $d(\widetilde{x'}) = (-1)^i x' \circ h^f - \overline{x'} \circ f - h^x$.
\end{itemize}
\[
\begin{tikzcd}
A \ar[r, "f"] \ar[rr, bend left=45, "0" name=zero] \ar[rrd, bend right=30, "0" name=zero', swap] & |[alias=A']| A' \ar[r, "c"] \ar[rd, "x"] & C \\
 & & X
 \ar[Rightarrow, from=A', to=zero, "h^f", shorten >=1mm]
 \ar[Rightarrow, from=A', to=zero', "h^x", swap, shorten <=1mm, shorten >=1mm]
\end{tikzcd}
\]
\end{definition}

\begin{proposition}
Consider the following $3$-term homotopy complex:
\[
\begin{tikzcd}
A \ar[r, "f"] \ar[rr, "0" name=zero, bend left=45] & |[alias=Y]| A' \ar[r, "c"] & C
\ar[Rightarrow, from=Y, to=zero, "h^f", shorten >=1mm]
\end{tikzcd}
\]
This is a homotopy cokernel of $f$ if and only if the following conditions are satisfied:
\begin{itemize}
    \item [(i)] Let $x \colon A' \to X$ be a closed morphism of degree $i$, and let $h^x \colon A \to X$ be a morphism of degree $i-1$ satisfying $d(h^x) = -x \circ f$. Then, there exists a closed morphism $x' \colon C \to X$ of degree $i$ that admits a cofiller.
    \item [(ii)] In (i), if both $x' \colon C \to X$ and $x'' \colon C \to X$ admits a cofiller, there exists a morphism $h \colon C \to X$ of degree $i-1$ such that $d(h) = x' - x''$.
\end{itemize}
\end{proposition}

\begin{definition}
A $3$-term homotopy complex is called a \it{homotopy short exact sequence}, or \it{homotopy short exact}, if it is both left exact and right exact:
\[
\begin{tikzcd}
X \ar[r, "f", swap] \ar[rr, "0", bend left=45, "{}" name=zero] & |[alias=Y]| Y \ar[r, "g", swap] & Z
\ar[Rightarrow, from=Y, to=zero, "h", shorten >=1mm]
\end{tikzcd}
\]
\end{definition}

From now on, we will explore basic properties of homotopy kernels. Note that the results for homotopy cokernels hold dually, up to signs.

\begin{lemma}\label{lem: induced morphism is null-homotopic}
In Proposition \ref{prop: univ of kernels} (i), the induced morphism $x'$ is null-homotopic if and only if there exists a pair of morphisms $(y, h_y)$ such that $x = d(y)$ and $h_x = -d(h_y) - f \circ y$.
\end{lemma}
\begin{proof}
By assumption, the following morphism is an isomorphism:
\[
H^i\left(\begin{bmatrix}k^\land\\h_f^\land\end{bmatrix}\right) \colon H^i(K^\land) \to H^i(\Cone f^\land[-1]).
\]
Therefore, $(x, h_x) \in Z^i(\Cone f^\land[-1])$ is zero in the $i$-th cohomology group if and only if the corresponding element $x' \in Z^i(K^\land)$ is zero in the $i$-th cohomology group.
\end{proof}

\begin{proposition}\label{prop: long exact induced by left exact}
Let the diagram \eqref{kernel} be a left exact 3-term complex. Then, the following long exact sequence in $\Mod H^0(\MA)$ is obtained:
\[
\begin{tikzcd}[row sep=0.3cm]
\cdots \ar[r] & H^i((\Ker f)^\land) \ar[r] & H^i(A^\land) \ar[r] & H^i(A'^\land) \\
\cdots \ar[r] & H^{-1}((\Ker f)^\land) \ar[r] & H^{-1}(A^\land) \ar[r] & H^{-1}(A'^\land) \\
\ar[r] & H^{0}((\Ker f)^\land) \ar[r] & H^{0}(A^\land) \ar[r] & H^{0}(A'^\land).
\end{tikzcd}
\]
\end{proposition}
\begin{proof}
Since $\RD(\MA)$ is a triangulated category, the following long exact sequence is obtained in general:
\[
\begin{tikzcd}[row sep=0.3cm]
\cdots \ar[r] & H^i(\Cone(f^\land[-1])) \ar[r] & H^i(A^\land) \ar[r] & H^i(A'^\land) \\
\cdots \ar[r] & H^{-1}(\Cone(f^\land[-1])) \ar[r] & H^{-1}(A^\land) \ar[r] & H^{-1}(A'^\land) \\
\ar[r] & H^{0}(\Cone(f^\land[-1])) \ar[r] & H^{0}(A^\land) \ar[r] & H^{0}(A'^\land).
\end{tikzcd}
\]
Thus, it suffices to show the commutativity of the following diagram, which is evident from the definition of the homotopy kernel:
\[
\begin{tikzcd}
& H^i(\Ker f^\land) \ar[ld, "H^i\left(\begin{bmatrix}k_f^\land\\h_f^\land\end{bmatrix}\right)", swap] \ar[rd, "H^i(k_f^\land)"] & \\
H^i(\Cone f^\land[-1]) \ar[rr, "H^i\left(\begin{bmatrix}
    \id_{A^\land}  0
\end{bmatrix}\right)", swap] & & H^i(A^\land).
\end{tikzcd}
\]
\end{proof}

\begin{proposition}\label{prop: torikae}
Consider the following diagram where $(k, f, h_f)$ is left exact:
\[
\begin{tikzcd}
K \ar[r, "k"] \ar[rr, "0" name=zero, bend left=45] & |[alias=Y]| A \ar[r, "f"] & A' \\
X \ar[ru, "0", swap] \ar[rru, bend right=30, "0" name=zero', swap] & &
\ar[Rightarrow, from=Y, to=zero, "h_f", shorten >=1mm]
\ar[Rightarrow, from=Y, to=zero', "0", swap, shorten <=1mm, shorten >=2mm]
\end{tikzcd}
\]
Let $x' \colon X \to K$ be a closed morphism of degree $i$ with a filler $(\overline{x'}, \widetilde{x'})$. Then, there exist a triple of morphisms $r \colon X \to K$, $\overline{r} \colon X \to A$, and $\widetilde{r} \colon X \to A'$ satisfying the following:
\begin{itemize}
    \item [(i)] $|r| = i-1$, $|\overline{r}| = i-2$, $|\widetilde{r}| = i-3$.
    \item [(ii)] $d(r) = x'$, $d(\overline{r}) = \overline{x'} + k \circ r$, $d(\widetilde{r}) = h_f \circ r - f \circ \overline{r} - \widetilde{x'}$.
\end{itemize}
\end{proposition}
\begin{proof}
By assumption, the two closed morphisms $0 \colon X \to K$ and $x' \colon X \to K$ of degree $i$ each have a filler. Thus, by the universality of the homotopy kernel, there exists a morphism $r' \colon X \to K$ of degree $i-1$ such that $d(r') = x'$. Using this, the pair of morphisms $(\overline{x'} + k \circ r', \widetilde{x'} + h_f \circ r')$ satisfies the following:
\begin{itemize}
    \item $d(\overline{x'} + k \circ r') = -k \circ x' + k \circ x' = 0$.
    \item $d(\widetilde{x'} + h_f \circ r') = h_f \circ x' - f \circ \overline{x'} - f \circ k \circ r' - h_f \circ x' = -f \circ (\overline{x'} + k \circ r')$.
\end{itemize}
By applying the universality of the homotopy kernel again, there exists a closed morphism $r'' \colon X \to K$ degree $i-1$ with a filler $(\overline{r}, \widetilde{r})$. The triple $(r' - r'', \overline{r}, \widetilde{r})$ satisfies the desired properties.
\end{proof}

\begin{lemma}\label{lem:mor extend to homotopy kernel}
Consider the following left exact $3$-term homotopy complex and an arbitrary $3$-term homotopy complex:
\[
\begin{tikzcd}
K \ar[r, "k", swap] \ar[rr, "0", bend left=45, "{}" name=zero] & |[alias=Y]| A \ar[r, "f", swap] & A' & &
X \ar[r, "i", swap] \ar[rr, "0", bend left=45, "{}" name=zero'] & |[alias=Y']| Y \ar[r, "j", swap] & Z
\ar[Rightarrow, from=Y, to=zero, "h_f", shorten >=1mm]
\ar[Rightarrow, from=Y', to=zero', "h", shorten >=1mm]
\end{tikzcd}
\]
If the following diagram exists, there is a morphism $(g_0, g_1, g_2, s, s', t)$ between the $3$-term homotopy complexes:
\[
\begin{tikzcd}
Y \ar[r, "j"] \ar[d, "g_1", swap] & |[alias=Z]| Z \ar[d, "g_2"] \\
|[alias=A]| A \ar[r, "f", swap] & A'
\ar[Rightarrow, "s'", from=Z, to=A, shorten <=2mm, shorten >=2mm]
\end{tikzcd}
\]
\end{lemma}
\begin{proof}
By assumption, the pair $(g_1 \circ i, g_2 \circ h - s' \circ i)$ satisfies $d(g_1 \circ i) = 0$ and the following equation:
\[
d(g_2 \circ h - s' \circ i) = -g_2 \circ j \circ i - (f \circ g_1 \circ i - g_2 \circ j \circ i) = -f \circ g_1 \circ i.
\]
Hence, by the universality of the homotopy kernel, there exists a morphism $g_0 \colon X \to K$ of degree 0 and its filler $(s'', t')$. Then, $(g_0, g_1, g_2, -s'', s', -t')$ is the desired morphism.
\end{proof}

\begin{lemma}\label{lem: splits are homotopy short exact}
A homotopy $3$-term complex that is a split short exact sequence in $H^0(\MA)$ is a homotopy short exact in $\MA$.
\end{lemma}

\begin{proof}
Let $(f, g, h)$ be a homotopy $3$-term complex satisfying the assumptions. Since the assumption is self-dual, it suffices to prove left exactness.

Since $g$ is a retraction in $H^0(\MA)$, it is also a retraction in $\RD(\MA)$. Thus, in $\RD(\MA)$, the following triangle exists:
\[
\begin{tikzcd}
X^\land \ar[r, "f^\land"] & Y^\land \ar[r, "g^\land"] & Z^\land \ar[r, "0", dashed] & {}
\end{tikzcd}
\]
This implies that the canonical morphism $X^\land \to \Cone(g^\land)[-1]$ is an isomorphism in $\RD(\MA)$. In particular, $(f, g, h)$ is left exact.
\end{proof}

\begin{definition}[{\cite[Definition 3.14]{chen2023exact}}]
Consider the following diagram:
\[
\begin{tikzcd}
B \ar[r, "g"] \ar[d, "a", swap] & |[alias=Z]| B' \ar[d, "a'"] \\
|[alias=A]| A \ar[r, "f", swap] & A'
\ar[Rightarrow, "s", from=Z, to=A, shorten <=2mm, shorten >=2mm]
\end{tikzcd}
\]
If the induced morphism $H^i([a^\land, g^\land, s^\land]^t) \colon H^i(B'^\land) \to H^i(\Cone([f^\land, -a'^\land])[-1])$ is an isomorphism for all $i \leq 0$, then the above diagram is called the \it{homotopy fiber product} of $f$ and $a'$, or a \it{homotopy Cartesian} diagram. The \it{homotopy fiber coproduct}  (\it{homotopy co-Cartesian diagram}) is defined dually.
\end{definition}

\begin{lemma}\label{lem: characterization of finite completeness}
If $\MA$ admits homotopy kernels and $H^0(\MA)$ is additive, then $\MA$ always has homotopy fiber products.
\end{lemma}
\begin{proof}
Consider the following diagram and let $X \in \MA$ be an object such that $X \cong A \oplus B'$ in $H^0(\MA)$. Let this isomorphism $X \cong A \oplus B'$ be induced by $p_A \colon X \to A$ and $p_{B'} \colon X \to B'$.
\[
\begin{tikzcd}
 & B' \ar[d, "a'"] \\
A \ar[r, "f"] & A'
\end{tikzcd}
\]
Then, we obtain a closed morphism $f \circ p_A - a' \circ p_{B'} \colon X \to A'$ of degree $0$. Let the following diagram the homotopy kernel of this morphism:
\[
\begin{tikzcd}
B \ar[r, "k"] \ar[rr, "0" name=zero, bend left=45] & |[alias=Y]| X \ar[r] & A'
\ar[Rightarrow, from=Y, to=zero, "s", shorten >=1mm]
\end{tikzcd}
\]
Set $a := p_A \circ k$ and $g := p_{B'} \circ k$, resulting in the following diagram:
\[
\begin{tikzcd}
B \ar[r, "g"] \ar[d, "a", swap] & |[alias=Z]| B' \ar[d, "a'"] \\
|[alias=A]| A \ar[r, "f", swap] & A'
\ar[Rightarrow, "s", from=Z, to=A, shorten <=2mm, shorten >=2mm]
\end{tikzcd}
\]
This diagram is homotopy Cartesian. Indeed, for any $i \leq 0$,
\begin{align*}
H^i(B^\land) & \cong H^i(\Cone(f \circ p_A - a' \circ p_{B'}^\land)[-1]) \\
& \cong H^i(\Cone([f^\land, -a'^\land])[-1]),
\end{align*}
where the second isomorphism is induced by the isomorphism $[p_A^\land, p_{B'}^\land]^t \colon X^\land \to A^\land \oplus B'^\land$ in $\RD(\MA)$.
\end{proof}

\begin{proposition}[{\cite[Proposition 4.9]{chen2023exact}}]
\label{prop: pullback of seq}
Consider the following $3$-term homotopy complexes and a morphism between them $(\id_A, b, c, s, s', t)$:
\[
\begin{tikzcd}
A \ar[r] \ar[d, equal] \ar[rr, bend left=45, "0" name=0] & |[alias=B_0]| B_0 \ar[r] \ar[d, "b"] & |[alias=C_0]| C_0 \ar[d, "c"] \\
|[alias=A]| A \ar[r] \ar[rr, bend right=45, "0" name=0', swap] & |[alias=B_1]| B_1 \ar[r] & C_1
\ar[Rightarrow, from=C_0, to=B_1, "s'"]
\ar[Rightarrow, from=B_0, to=A, "s"]
\ar[Rightarrow, from=B_0, to=0, shorten >=1mm]
\ar[Rightarrow, from=B_1, to=0', shorten >=1mm]
\end{tikzcd}
\]
Then, the following statements hold:
\begin{itemize}
    \item [(i)] If the top row is right exact, then the bottom row is right exact if and only if the right square is homotopy co-Cartesian.
    \item [(ii)] If the right square is homotopy Cartesian, then the top row is left exact if and only if the bottom row is left exact.
\end{itemize}
\end{proposition}

\begin{proposition}[{\cite[Lemma 4.16]{chen2023exact}}]
\label{prop: octahedoron of left exact}
Let the following two $3$-term homotopy complexes be left exact:
\[
\begin{tikzcd}
\Ker f \ar[r, "k_f", swap] \ar[rr, "0", bend left=45, "{}" name=zero] & |[alias=Y]| A \ar[r, "f", swap] & A' \quad\,
& & \Ker f' \ar[r, "k_{f'}", swap] \ar[rr, "0", bend left=45, "{}" name=zero'] & |[alias=Y']| A' \ar[r, "f'", swap] & A'' \quad\,\\
\ar[Rightarrow, from=Y, to=zero, "h_f", shorten >=1mm]
\ar[Rightarrow, from=Y', to=zero', "h_{f'}", shorten >=1mm]
\end{tikzcd}
\]
Assume the homotopy fiber product of $A \to A' \leftarrow \Ker f'$ exists. Then, the following diagram exists:
\[
\begin{tikzcd}
\Ker f \ar[r, equal] \ar[d, "u"] & \Ker f \ar[d, "k_f"] & \\
K \ar[r, "k"] \ar[ru, Rightarrow, "s", shorten <=3mm, shorten >=3mm] \ar[d, "v"] & A \ar[r, "f' \circ f"] \ar[d, "f"] & A'' \ar[d, equal] \\
\Ker f' \ar[r, "k_{f'}"] \ar[ru, Rightarrow, "s'", shorten <=3mm, shorten >=3mm] & A' \ar[ru, Rightarrow, "0", shorten <=3mm, shorten >=3mm] \ar[r, "f'"] & A''
\end{tikzcd}
\]
Here, the bottom-left square is a homotopy fiber product, and the leftmost column $(u, v, h)$ is left exact, where $h$ is a degree $-1$ morphism $h \colon \Ker f \to \Ker f'$. Furthermore, the following morphisms between the $3$-term homotopy complexes exist:
\begin{itemize}
    \item [(i)] There exists $t \colon \Ker f \to A'$ such that $(\id_{\Ker f}, k, k_{f'}, s, s', t)$ is a morphism from $(u, v, h)$ to $(k_f, f, h_f)$.
    \item [(ii)] Let $h'' := h_{f'} \circ v - f' \circ s'$. Then, $(v, f, \id_{A''}, -s', 0, 0)$ is a morphism from $(k, f' \circ f, h'')$ to $(k_{f'}, f', h_{f'})$.
    \item [(iii)] There exists $t' \colon \Ker f \to A''$ such that $(u, \id_A, f', -s, 0, t')$ is a morphism from $(k_f, f, h_f)$ to $(k, f' \circ f, h'')$.
\end{itemize}
\end{proposition}
\begin{proof}
Let $K$ be the homotopy fiber product of $A \rightarrow A' \leftarrow \Ker f'$. 
In this situation, we obtain a homotopy triple complex $(u,v,h)$ and a morphism 
$(\id_{\Ker f},k,k_{f'},s,s',t)$ from $(u,v,h)$ to $(k_f,f,h_f)$. 
By Proposition \ref{prop: pullback of seq}, the triple complex $(u,v,h)$ forms a left exact sequence. 
Regarding the morphisms between homotopy triple complexes, (i) and (iii) are induced by homotopy kernels, 
so there is no issue. One can verify (ii) by computing the differentials.
\end{proof}

Exact DG categories were introduced in \cite{chen2023exact}. The homotopy categories of exact DG categories naturally form extriangulated categories. Since abelian $n$-truncated DG-categories introduced in this work inherit the structure of exact DG categories, we provide their definitions and discuss their properties.

% Exact DG categories are introduced in \cite{chen2023exact}. Homotopy categories of exact DG categories natulally become extriangulated categories.  Since the abelian $(n,1)$-categories introduced here naturally inherit the structure of complete DG categories, we present their definitions and properties.

\begin{definition}[{\cite[Definition 5.1]{chen2023exact}}]
\label{definition: exact dg}
Let $\MA$ be a DG-category, and let $\MS \subset \mathcal{H}_{\mathrm{3t}}(\MA)$ be a subcategory closed under isomorphisms. For any $(f, g, h) \in \MS$, we call $f$ an inflation and $g$ a deflation. The pair $(\MA, \MS)$ is called a \it{exact DG-category} if it satisfies the following conditions:
\begin{itemize}
    \item [(Ex0):] $\id_A$ is a deflation for all $A \in \MA$.
    \item [(Ex1):] Deflations are closed under composition.
    \item [(Ex2):] For any deflation $g \colon B' \to C'$, the homotopy fiber product of the diagram $B' \xrightarrow{g} C' \xleftarrow{c} C$ exists, and deflations are preserved under homotopy fiber products.
    \item [(Ex2)':] The dual of (Ex2).
\end{itemize}
\end{definition}

\begin{theorem}[{\cite[Theorem 6.19]{chen2023exact}}]
\label{thm: enhance of et}
Let $(\MA, \MS)$ be an exact DG-category. Then, $H^0(\MA)$ naturally inherits the structure of an extriangulated category, where the conflations in this extriangulated category correspond to the objects in $\MS$.
\end{theorem}

\begin{proposition}[{\cite[Example-Definition 5.7]{chen2023exact}}]
\label{prop: ect-cl exact dg}
Let $(\MA, \MS)$ be an exact DG-category. If a full DG subcategory $\MA' \subset \MA$ is closed under homotopy equivalence and extensions, i.e., if the middle object of any conflation whose both ends are in $\MA'$ also belongs to $\MA'$, then $(\MA', \MS')$ is also an exact DG-category. Here, $\MS'$ is the subcategory of $\MS$ consisting of conflations where all three objects belong to $\MA'$.
\end{proposition}

\section{preabelian DG-categories}\label{sec: preabel}

Before discussing abelian $n$-truncated DG-categories, we introduce preabelian DG-categories. An additive category $\MA$ is a preabelian DG-category if it is a preabelian category in the usual sense. This section's main theorem asserts that the homotopy category of a preabelian DG-category naturally inherits a pretriangulated structure. 

\subsection{Loop functors and suspension functors}

In this subsection, we study the basic properties of the loop functor and the suspension functor, which are essential for defining a pretriangulated structure.

\begin{definition}
Let $\MA$ be a DG-category.
\begin{itemize}
    \item If $H^0(\MA)$ is additive, $\MA$ is called an additive DG-category.
    \item If $\MA$ is additive and admits homotopy kernels (resp. homotopy cokernels), $\MA$ is called \it{finitely complete} (resp. \it{finitely cocomplete}).
    \item If $\MA$ is both finitely complete and finitely cocomplete, $\MA$ is called a preabelian DG-category.
\end{itemize}
\end{definition}

\begin{definition}
Let $\MA$ be a DG-category, and let $X \in \MA$. If $\Omega X \in \MA$ and a closed morphism $h_X \colon \Omega X \to X$ of degree $-1$ induce the following isomorphism for all $i \leq 0$, then $(\Omega X, h_X)$ is called the \it{loop object} of $X$:
\[
H^i(s^{-}_{X^\land} \circ ({h_X}^\land)) \colon H^i(\Omega X^\land) \to H^i(X^\land[-1]),
\]
where $s_{X^\land}^{-} \colon X^\land \to X^\land[-1]$ is a shift morphism of degree $1$.
\end{definition}

The suspension object is defined dually.

\begin{definition}
Let $\MA$ be a DG-category, and let $X \in \MA$. If $\Sigma X \in \MA$ and a closed morphism $h^X \colon X \to \Sigma X$ of degree $-1$ induce the following isomorphism for all $i \leq 0$, then $(\Sigma X, h^X)$ is called the \it{suspension object} of $X$:
\[
H^i(s^{-}_{X^\lor} \circ ({h^X}^\lor)) \colon H^i(\Sigma X^\lor) \to H^i(X^\lor[-1]).
\]
\end{definition}

\begin{remark}
Loop objects and suspension objects, if they exist, are unique up to homotopy equivalence.
\end{remark}

The universality of loop objects can be described as follows:

\begin{proposition}
For $X \in \MA$, let $X' \in \MA$ and a closed morphism $h \colon X' \to X$ of degree $-1$ be given. Then, the following are equivalent:
\begin{itemize}
    \item [(i)] $(X', h)$ is the loop object of $X$.
    \item [(ii)]
    \begin{itemize}
        \item [(1)] For any morphism $f \colon Y \to X$ of degree $i$, there exist a degree $i+1$ closed morphism $f' \colon Y \to X'$ and a degree $i-1$ morphism $\overline{f'} \colon Y \to X$ such that $d(\overline{f'}) = f - h \circ f'$.
        \item [(2)] If $f'_0, f'_1 \colon Y \to X'$ satisfy (1), there exists a morphism $r \colon Y \to X'$ of degree $i$ such that $d(r) = f'_1 - f'_0$.
    \end{itemize}
\end{itemize}
\end{proposition}
\begin{proof}
Immediately from the definition of loop objects.
\end{proof}

\begin{definition}
Suppose $\MA$ has loop objects for all objects. Then, the functor $\Omega \colon H^0(\MA) \to H^0(\MA)$ is defined as follows:
\begin{itemize}
    \item $\Omega$ maps $X$ to $\Omega X$ (choosing a representative from the isomorphism class of $\Omega X$).
    \item The morphism $H^0(\MA)(X, Y) \to H^0(\MA)(\Omega X, \Omega Y)$ is defined as the composition:
    \[
    H^0(\MA(X, Y)) \xrightarrow{H^0(s^{-}_{(\Omega X)^\lor} \circ (h_X)^\lor)} H^0(\MA(\Omega X, Y)[-1]) \xrightarrow{H^0(s^{-}_{(\Omega X)^\lor} \circ (h_Y)^\land)^{-1}} H^0(\MA(\Omega X, \Omega Y)).
    \]
\end{itemize}
\end{definition}

\begin{remark}\label{rem: Omega corresponds}
From the construction of $\Omega$, $f' = \Omega f$ in $H^0(\MA)$ if and only if there exists a degree $-2$ morphism $t \colon \Omega X \to Y$ such that $d(t) = h_Y \circ f' - f \circ h_X$. Thus, verifying that $\Omega$ is a functor is straightforward.
\end{remark}

\begin{remark}\label{rem: omega is rep obj}
The above definition immediately implies the following isomorphism in $\Mod H^0(\MA)$: $\Omega X$ represents $H^{-1}(\MA(-, X))$ in $H^0(\MA)$:
\[
H^0(\MA(-, \Omega X)) \cong H^{-1}(\MA(-, X)).
\]
\end{remark}

Dually, we define the suspension functor.

\begin{definition}
Assume that $\MA$ has suspension objects for all objects. The functor $\Sigma \colon H^0(\MA) \to H^0(\MA)$ is defined as follows:
\begin{itemize}
    \item $\Sigma$ maps $X$ to $\Sigma X$ (choosing a representative from the isomorphism class of $\Sigma X$).
    \item The morphism $H^0(\MA)(X, Y) \to H^0(\MA)(\Sigma X, \Sigma Y)$ is defined as the composition:
    \[
    H^0(\MA(X, Y)) \xrightarrow{H^0(s^{-}_{(\Sigma Y)^\land} \circ (h^Y)^\land)} H^0(\MA(X, \Sigma Y)[-1]) \xrightarrow{H^0(s^{-}_{(\Sigma Y)^\land} \circ (h^X)^\lor)^{-1}} H^0(\MA(\Sigma X, \Sigma Y)).
    \]
\end{itemize}
\end{definition}

\begin{proposition}\label{prop: adjoint of omega}
If $\MA$ has loop objects and suspension objects for all objects, then $\Sigma \dashv \Omega$ holds in $H^0(\MA)$.
\end{proposition}
\begin{proof}
Consider the following isomorphism, which we show establishes the adjunction:
\[
\phi := H^0(\MA(\Sigma X, Y)) \xrightarrow{H^0(s^{-}_Y \circ (h^X)^\lor)} H^0(\MA(X, Y)[-1]) \xrightarrow{-H^0(s^{-}_Y \circ (h_Y)^\land)^{-1}} H^0(\MA(X, \Omega Y)).
\]

By the definition of $\phi$, for any $f \in Z^0(\MA)(\Sigma X, Y)$ and $f' \in Z^0(\MA)(X, \Omega Y)$, $\phi(f) = f'$ if and only if there exists a degree $-2$ morphism $t \colon X \to Y$ such that $d(t) = f \circ h^X + h_Y \circ f'$.

We now verify the naturality of $\phi$. By duality, it suffices to show that for any closed morphism $g \colon \Sigma X' \to Y$ and $f \colon X \to X'$ of degree $0$, $\phi(g \circ \Sigma f) = \phi(g) \circ f$ holds. Consider the following diagram:
\[
\begin{tikzcd}
X \ar[r, "f"] \ar[d, "h^X"] & X' \ar[r, "\phi(g)"] \ar[d, "h^{X'}"] & \Omega Y \ar[d, "h_Y"] \\
\Sigma X \ar[r, "\Sigma f"] & \Sigma X' \ar[r, "g"] & Y
\end{tikzcd}
\]

By Remark \ref{rem: Omega corresponds}, there exists a degree $-2$ morphism $t \colon X \to \Sigma X'$ such that $d(t) = \Sigma f \circ h^X - h^{X'} \circ f$. Additionally, by the definition of $\phi$, there exists a degree $-2$ morphism $t' \colon X' \to Y$ such that $d(t') = g \circ h^{X'} + h_Y \circ \phi(g)$. Therefore, the morphism $g \circ t + t' \circ f \colon X \to Y$ satisfies $d(g \circ t + t' \circ f) = h_Y \circ \phi(g) \circ f + g \circ \Sigma f \circ h^X$, proving $\phi(g \circ \Sigma f) = \phi(g) \circ f$.
\end{proof}

\begin{remark}\label{rem: adjunction corresponds}
As shown in the proof of Proposition \ref{prop: adjoint of omega}, $\phi(f) = f'$ if and only if there exists a degree $-2$ morphism $t \colon X \to Y$ such that $d(t) = f \circ h^X + h_Y \circ f'$.
\end{remark}

\begin{proposition}\label{prop: preabelian admits loop}
If $\MA$ is a finitely complete DG-category, then $\MA$ has loop objects for all objects.
\end{proposition}
\begin{proof}
For $X \in \MA$, consider the zero morphism $I \xrightarrow{0} X$ from a contractible object $I$. Take its homotopy kernel:
\[
\begin{tikzcd}
X' \ar[r, "k", swap] \ar[rr, "0", bend left=45, "{}" name=zero] & |[alias=Y]| I \ar[r, "0", swap] & X
\ar[Rightarrow, from=Y, to=zero, "h_X", shorten >=1mm]
\end{tikzcd}
\]

By the definition of the homotopy kernel, $d(h_X) = 0$, and $h_X$ induces the isomorphism 
$$H^i(s^{-}_{X^\land} \circ (h_X^\land)) \colon H^i(X'^\land) \to H^i(X^\land[-1])$$
for all $i \leq 0$. Hence, $(X', h_X)$ is the loop object of $X$.
\end{proof}

\subsection{preabelian DG-categories and their homotopy categories}

In this section, we aim to prove the following statement:

\begin{theorem}\label{thm: pretriang}
Let $\MA$ be a preabelian DG-category. Then, $H^0(\MA)$ naturally inherits a pretriangulated structure $(\Omega, \Sigma, \Delta, \nabla)$ (see Definition \ref{def: pretr}).
\end{theorem}

We first discuss left triangles in $H^0(\MA)$. Let $f_0 \colon A_1 \to A_0$ be a degree 0 closed morphism. By Proposition \ref{prop: pullback of seq} (ii), we obtain the following diagram and a morphism between $3$-term homotopy complexes $(\id_{\Omega A_0}, i, f_0, q, h_{f_0}, \eta_{f_0})$. Here, $I$ is a contractible object:
\[
\begin{tikzcd}
\Omega A_0 \ar[r, "f_2"] \ar[d, equal] \ar[rr, bend left=45, "0" name=0] & |[alias=B_0]| A_2 \ar[r, "f_1"] \ar[d, "i"] & |[alias=C_0]| A_1 \quad \ar[d, "f_0"] \\
|[alias=A]|\Omega A_0 \ar[r] \ar[rr, bend right=45, "0" name=0', swap] & |[alias=B_1]| I \ar[r, "0", swap] & A_0 \quad
\ar[Rightarrow, from=C_0, to=B_1, "h_{f_0}"]
\ar[Rightarrow, from=B_0, to=A, "q"]
\ar[Rightarrow, from=B_0, to=0, shorten >=1mm, "h_{f_1}"]
\ar[Rightarrow, from=B_1, to=0', shorten >=1mm, "h_{A_0}"]
\end{tikzcd}
\]

Here, the lower row and the upper row are left exact, and the right square is homotopy Cartesian. Since $I$ is contractible, the $3$-term homotopy complex $(f_0, f_1, h_{f_0})$ is left exact.

The diagram $\Omega A_0 \to A_2 \to A_1 \to A_0$ in $H^0(\MA)$ obtained from this is called the standard left triangle associated with $f_0$. The collection of all triangles in $H^0(\MA)$ isomorphic to these standard left triangles is denoted by $\Delta$.

\begin{remark}
In the above diagram, $d(\eta_{f_0}) = f_0 \circ h_{f_1} - h_{f_0} \circ f_2 - h_{A_0}$.
\end{remark}

Dually, the standard right triangle associated with $f_0 \colon A_0 \to A_1$, denoted as $A_0 \xrightarrow{f_0} A_1 \xrightarrow{f_1} A_2 \xrightarrow{f_2} \Sigma A_0$, is defined by the following diagram:
\[
\begin{tikzcd}
\quad A_0 \ar[r, "0"] \ar[d, "f_0"] \ar[rr, bend left=45, "0" name=0] & |[alias=B_0]| I' \ar[r] \ar[d, "i'"] & |[alias=C_0]| \Sigma A_0 \ar[d, equal] \\
|[alias=A]|\quad A_1 \ar[r, "f_1", swap] \ar[rr, bend right=45, "0" name=0', swap] & |[alias=B_1]| A_2 \ar[r, "f_2", swap] & \Sigma A_0
\ar[Rightarrow, from=C_0, to=B_1, "q'"]
\ar[Rightarrow, from=B_0, to=A, "-h^{f_0}"]
\ar[Rightarrow, from=B_0, to=0, shorten >=1mm, "h^{A_0}"]
\ar[Rightarrow, from=B_1, to=0', shorten >=1mm, "h^{f_1}"]
\end{tikzcd}
\]

Here, both the upper and lower rows are right exact, and the right square is homotopy co-Cartesian. Moreover, $I'$ is a contractible object. There exists a degree $-2$ morphism $\eta^{f_0} \colon A_0 \to \Sigma A_0$, and the morphism $(f_0, i', \id_{\Sigma A_0}, -h^{f_0}, q', \eta^{f_0})$ forms a morphism of $3$-term homotopy complexes. Using these, we can define $\nabla$ similarly to $\Delta$.

\begin{remark}
In the above diagram, $\eta^{f_0} \colon A_0 \to \Sigma A_0$ satisfies $d(\eta^{f_0}) = h^{A_0} + f_2 \circ h^{f_0} - h^{f_1} \circ f_0$.
\end{remark}

\begin{lemma}\label{lem: LT3}
Let the left diagram represent a left exact $3$-term homotopy complex and a morphism between such complexes $(a_2, a_1, a_0, s', s, t)$. Let the right diagram represent a $3$-term homotopy complex and a morphism $(a', a_2, a_1, s'', s', t')$, obtained by the universality of the homotopy kernel. Then, in $H^0(\MA)$, we have $a' = \Omega a_0$.
\[
\begin{tikzcd}
A_2 \ar[r, "f_1"] \ar[d, "a_2"] \ar[rr, bend left=45, "0" name=0] & |[alias=B_0]| A_1 \ar[r, "f_0"] \ar[d, "a_1"] & |[alias=C_0]| A_0 \ar[d, "a_0"] \\
|[alias=A]| B_2 \ar[r, "g_1", swap] \ar[rr, bend right=45, "0" name=0', swap] & |[alias=B_1]| B_1 \ar[r, "g_0", swap] & B_0
\ar[Rightarrow, from=C_0, to=B_1, "s"]
\ar[Rightarrow, from=B_0, to=A, "s'"]
\ar[Rightarrow, from=B_0, to=0, shorten >=1mm, "h_{f_0}"]
\ar[Rightarrow, from=B_1, to=0', shorten >=1mm, "h_{g_0}"]
\end{tikzcd}
\quad
\begin{tikzcd}
\Omega A_0 \ar[r, "f_2"] \ar[d, "a'", dashed] \ar[rr, bend left=45, "0" name=0]
& |[alias=B_0]| A_2 \ar[r, "f_1"] \ar[d, "a_2"]
& |[alias=C_0]| A_1 \, \ar[d, "a_1"] \\
|[alias=A]| \Omega B_0 \ar[r, "g_2", swap] \ar[rr, bend right=45, "0" name=0', swap]
& |[alias=B_1]| B_2 \ar[r, "g_1", swap]
& B_1 \,
\ar[Rightarrow, from=C_0, to=B_1, "s'"]
\ar[Rightarrow, from=B_0, to=A, "s''"]
\ar[Rightarrow, from=B_0, to=0, shorten >=1mm, "h_{f_1}"]
\ar[Rightarrow, from=B_1, to=0', shorten >=1mm, "h_{g_1}"]
\end{tikzcd}
\]
\end{lemma}

\begin{proof}
Note that the degree $-2$ morphism $t' \colon \Omega A_0 \to B_1$ satisfies $d(t') = a_1 \circ h_{f_1} - s' \circ f_2 - g_1 \circ s'' - h_{g_1} \circ a'$. By Remark \ref{rem: Omega corresponds}, it suffices to show that there exists a degree $-2$ morphism $\widetilde{\eta} \colon \Omega A_0 \to B_2$ such that $d(\widetilde{\eta}) = h_{B_0} \circ a' - a_0 \circ h_{A_0}$. Define $\widetilde{\eta}$ as follows:
\[
\widetilde{\eta} = a_0 \circ \eta_{f_0} + t \circ f_2 + s \circ h_{f_1} - g_0 \circ t' + h_{g_0} \circ s'' - \eta_{g_0} \circ a'.
\]

We verify that $d(\widetilde{\eta}) = h_{B_0} \circ a' - a_0 \circ h_{A_0}$ through the following calculation:
\begin{align*}
    d(a_0 \circ \eta_{f_0} + t \circ f_2 + s \circ h_{f_1}& - g_0 \circ t' + h_{g_0} \circ s'' - \eta_{g_0} \circ a') \\
    =& \, a_0 \circ (f_0 \circ h_{f_1} - h_{f_0} \circ f_2 - h_{A_0}) \\
    &+ (a_0 \circ h_{f_0} - s \circ f_1 - g_0 \circ s' - h_{g_0} \circ a_2) \circ f_2 \\
    &+ (g_0 \circ a_1 - a_0 \circ f_0) \circ h_{f_1} - s \circ (-f_1 \circ f_2) \\
    &- g_0 \circ (a_1 \circ h_{f_1} - s' \circ f_2 - g_1 \circ s'' - h_{g_1} \circ a') \\
    &+ (-g_0 \circ g_1) \circ s'' - h_{g_0} \circ (g_2 \circ a' - a_2 \circ f_2) \\
    &- (g_0 \circ h_{g_1} - h_{g_0} \circ g_2 - h_{B_0}) \circ a' \\
    =& \, h_{B_0} \circ a' - a_0 \circ h_{A_0}.
\end{align*}
\end{proof}

\begin{proposition}\label{prop: left triang}
Let $\MA$ be a finitely complete DG-category. Then, $(H^0(\MA), \Omega, \Delta)$ is a left triangulated category (cf. Definition \ref{def: left triangulated cat}).
\end{proposition}
\begin{proof}
We verify the axioms one by one.

(LT1): For any morphism $f \colon A \to A'$ in $H^0(\MA)$, a standard left triangle exists by construction.  
Moreover, if $I$ is a contractible object and $A \in \MA$, the homotopy kernel of $A \to I$ is given by $A \xrightarrow{\id_A} A$ (Lemma \ref{lem: splits are homotopy short exact}). Thus, $0 \to A = A \to 0$ forms a left triangle in $H^0(\MA)$.

(LT2): Let $f_0$ be a closed morphism of degree $0$, and consider the following two diagrams:

\[
\begin{tikzcd}
\Omega A_0 \ar[r, "f_2"] \ar[d, equal] \ar[rr, bend left=45, "0" name=0]
& |[alias=B_0]| A_2 \ar[r, "f_1"] \ar[d, "i"]
& |[alias=C_0]| A_1 \, \ar[d, "f_0"] \\
|[alias=A]|\Omega A_0 \ar[r] \ar[rr, bend right=45, "0" name=0', swap]
& |[alias=B_1]| I \ar[r, "0", swap]
& A_0 \,
\ar[Rightarrow, from=C_0, to=B_1, "h_{f_0}"]
\ar[Rightarrow, from=B_0, to=A, "q"]
\ar[Rightarrow, from=B_0, to=0, shorten >=1mm, "h_{f_1}"]
\ar[Rightarrow, from=B_1, to=0', shorten >=1mm, "h_{A_0}"]
\end{tikzcd}
\quad
\begin{tikzcd}
\Omega A_1 \ar[r, "f_3"] \ar[d, equal] \ar[rr, bend left=45, "0" name=0]
& |[alias=B_0]|\Omega A_0 \ar[r, "f_2"] \ar[d, "i'"]
& |[alias=C_0]| A_2 \, \ar[d, "f_1"] \\
|[alias=A]|\Omega A_1 \ar[r] \ar[rr, bend right=45, "0" name=0', swap]
& |[alias=B_1]| I' \ar[r, "0", swap]
& A_1 \,
\ar[Rightarrow, from=C_0, to=B_1, "h_{f_1}"]
\ar[Rightarrow, from=B_0, to=A, "q'"]
\ar[Rightarrow, from=B_0, to=0, shorten >=1mm, "h_{f_2}"]
\ar[Rightarrow, from=B_1, to=0', shorten >=1mm, "h_{A_1}"]
\end{tikzcd}
\]

Here, $(\id_{\Omega A_0}, i, f_0, q, h_{f_0}, \eta_{f_0})$ and $(\id_{\Omega A_1}, i', f_1, q', h_{f_1}, \eta_{f_1})$ are morphisms of $3$-term homotopy complexes. Thus, the following equations hold:
\begin{itemize}
    \item $d(\eta_{f_0}) = f_0 \circ h_{f_1} - h_{f_0} \circ f_2 - h_{A_0}$
    \item $d(\eta_{f_1}) = f_1 \circ h_{f_2} - h_{f_1} \circ f_3 - h_{A_1}$
\end{itemize}

We obtain two standard triangles:
\[
\Omega A_0 \xrightarrow{f_2} A_2 \xrightarrow{f_1} A_1 \xrightarrow{f_0} A_0,
\quad
\Omega A_1 \xrightarrow{f_3} \Omega A_0 \xrightarrow{f_2} A_2 \xrightarrow{f_1} A_1.
\]

It suffices to show that $-f_3 = \Omega f_0$ in $H^0(\MA)$. Consider the morphism $f_0 \circ \eta_{f_1} + \eta_{f_0} \circ f_3 + h_{f_0} \circ h_{f_2}$ of degree $-2$. Its differential is computed as follows:
\begin{align*}
d(f_0 \circ \eta_{f_1} + \eta_{f_0} \circ f_3 + h_{f_0} \circ h_{f_2})=& \, f_0 \circ f_1 \circ h_{f_2} - f_0 \circ h_{f_1} \circ f_3 - f_0 \circ h_{A_1} \\
 &+ f_0 \circ h_{f_1} \circ f_3 - h_{f_0} \circ f_2 \circ f_3 - h_{A_0} \circ f_3 \\
 &- f_0 \circ f_1 \circ h_{f_2} + h_{f_0} \circ f_2 \circ f_3 \\
=& -h_{A_0} \circ f_3 - f_0 \circ h_{A_1}.
\end{align*}
Thus, by Remark \ref{rem: Omega corresponds}, $-f_3 = \Omega f_0$ in $H^0(\MA)$.

(LT3): This follows immediately from Lemma \ref{lem: LT3}.

(LT4): For two morphisms $f \colon B \to A$ and $g \colon C \to B$, their standard left triangles are as follows:
\[
\Omega A \xrightarrow{f''} F \xrightarrow{f'} B \xrightarrow{f} A, \quad
\Omega B \xrightarrow{g''} E \xrightarrow{g'} C \xrightarrow{g} B.
\]

Here, $(f', f, h_f)$ and $(g', g, h_g)$ are left exact $3$-term homotopy complexes:
\[
\begin{tikzcd}
F \ar[r, "f'", swap] \ar[rr, "0", bend left=45, "{}" name=zero] & |[alias=Y]| B \ar[r, "f", swap] & A
\ar[Rightarrow, from=Y, to=zero, "h_f", shorten >=1mm]
\end{tikzcd}
\quad
\begin{tikzcd}
E \ar[r, "g'", swap] \ar[rr, "0", bend left=45, "{}" name=zero] & |[alias=Y]| C \ar[r, "g", swap] & B
\ar[Rightarrow, from=Y, to=zero, "h_g", shorten >=1mm]
\end{tikzcd}
\]

By Proposition \ref{prop: octahedoron of left exact}, we obtain the following diagram:
\[
\begin{tikzcd}
E \ar[r, "u'"] \ar[d, equal] & D \ar[r, "u"] \ar[d, "v'"] & F \ar[d, "f'"] \\
E \ar[r, "g'"] & C \ar[r, "g"] \ar[d, "f \circ g"] & B \ar[d, "f"] \\
& A \ar[r, equal] & A
\end{tikzcd}
\]

There exists $h_u \colon E \to F$ such that $(u', u, h_u)$ forms a left exact $3$-term homotopy complex. Thus, we obtain the standard left triangles $\Omega F \xrightarrow{u''} E \xrightarrow{u'} D \xrightarrow{u} F$ and $\Omega A \xrightarrow{v''} D \xrightarrow{v'} C \xrightarrow{f \circ g} A$.  
Applying Lemma \ref{lem: LT3} to the three morphisms from Proposition \ref{prop: octahedoron of left exact}, the following diagrams in $H^0(\MA)$ commute:
\[
\begin{tikzcd}
\Omega A \ar[d, "v''"] \ar[r, equal] & \Omega A \ar[d, "f''"] &  
\Omega F \ar[d, "\Omega f'"] \ar[r, "u''"] & E \ar[d, equal] &  
\Omega B \ar[d, "g''"] \ar[r, "\Omega f"] & \Omega A \ar[d, "v''"] \\
D \ar[r, "u"] & F &  
\Omega B \ar[r, "g''"] & E &  
E \ar[r, "u'"] & D.
\end{tikzcd}
\]

Thus, (LT4) is verified.
\end{proof}

Dually, we also have the following proposition:

\begin{proposition}\label{prop: right triang}
Let $\MA$ be a finitely cocomplete DG-category. Then, $(H^0(\MA), \Sigma, \nabla)$ forms a right triangulated category.
\end{proposition}

We are now ready to prove Theorem \ref{thm: pretriang}.

\begin{proof}[Proof of Theorem \ref{thm: pretriang}]
It suffices to verify the third axiom of Definition \ref{def: pretr}.

Consider the following diagram in $H^0(\MA)$:
\[
\begin{tikzcd}
A \ar[r, "f_0"] \ar[d, "\overline{\gamma}" swap] & B \ar[r, "f_1"] & C \ar[r, "f_2"] \ar[d, "\beta"] & \Sigma A \ar[d, "\gamma"] \\
\Omega C' \ar[r, "f'_0"] & A' \ar[r, "f'_1"] & B' \ar[r, "f'_2"] & C'
\end{tikzcd}
\]

Here, $\overline{\gamma}$ is the morphism corresponding to $\gamma$ under the adjunction $\Sigma \dashv \Omega$. The upper row is a standard right triangle, the lower row is a standard left triangle, and the right square commutes in $H^0(\MA)$.

Since the right square commutes, there exists a morphism $s \colon C \to C'$ of degree $-1$ such that $d(s) = f'_2 \circ \beta - \gamma \circ f_2$. By Lemma \ref{lem:mor extend to homotopy kernel}, there exist morphisms $\alpha \colon B \to A'$, $s' \colon B \to B'$, and $t \colon B \to C'$ satisfying:
\begin{itemize}
    \item $d(s') = f'_1 \circ \alpha - \beta \circ f_1$,
    \item $d(t) = \gamma \circ h^{f_1} - s \circ f_1 - f'_2 \circ s' - h_{f'_2} \circ \alpha$.
\end{itemize}

Next, consider the following diagram:
\[
\begin{tikzcd}
B \ar[r, "f_1"] \ar[d, "\alpha", swap] & |[alias=Z]| C \ar[d, "\beta"] \\
|[alias=A]| A' \ar[r, "f'_1", swap] & B'
\ar[Rightarrow, "s'", from=Z, to=A, shorten <=2mm, shorten >=2mm]
\end{tikzcd}
\]

Applying Lemma \ref{lem:mor extend to homotopy kernel} again, there exist a closed morphism $\gamma' \colon A \to \Omega C'$ of degree $0$ and morphisms $s'' \colon A \to A'$ and $t' \colon A \to B'$ such that:
\begin{itemize}
    \item $d(s'') = f'_0 \circ \gamma' - \alpha \circ f_0$,
    \item $d(t') = \beta \circ h^{f_0} - s' \circ f_0 - f'_1 \circ s'' - h_{f'_1} \circ \gamma'$.
\end{itemize}

It remains to show that $\gamma' = \overline{\gamma}$ in $H^0(\MA)$.

From the definitions of the standard left and right triangles, note that there exist morphisms $\eta^{f_0} \colon A \to \Sigma A$ and $\eta_{f'_2} \colon \Omega C' \to C'$ of degree $-2$ satisfying:
\begin{itemize}
    \item $d(\eta^{f_0}) = h^A + f_2 \circ h^{f_0} - h^{f_1} \circ f_0$,
    \item $d(\eta_{f'_2}) = f'_2 \circ h_{f'_1} - h_{f'_2} \circ f'_0 - h_{C'}$.
\end{itemize}

Consider the morphism of degree $-2$ :
\[
\gamma \circ \eta^{f_0} + t \circ f_0 + s \circ h^{f_0} - f'_2 \circ t' + h_{f'_2} \circ s'' - \eta_{f'_2} \circ \gamma'.
\]

Its differential is computed as follows:
\begin{align*}
d(\gamma \circ \eta^{f_0} + t \circ f_0 + s \circ h^{f_0} &- f'_2 \circ t' + h_{f'_2} \circ s'' - \eta_{f'_2} \circ \gamma') \\
=& \, \gamma \circ (h^A + f_2 \circ h^{f_0} - h^{f_1} \circ f_0) \\
 &+ (\gamma \circ h^{f_1} - s \circ f_1 - f'_2 \circ s' - h_{f'_2} \circ \alpha) \circ f_0 \\
 &+ (f'_2 \circ \beta - \gamma \circ f_2) \circ h^{f_0} - s \circ (-f_1 \circ f_0) \\
 &- f'_2 \circ (\beta \circ h^{f_0} - s' \circ f_0 - f'_1 \circ s'' - h_{f'_1} \circ \gamma') \\
 &+ (-f'_2 \circ f'_1) \circ s'' - h_{f'_2} \circ (f'_0 \circ \gamma' - \alpha \circ f_0) \\
 &- (f'_2 \circ h_{f'_1} - h_{f'_2} \circ f'_0 - h_{C'}) \circ \gamma' \\
=& \, h_{C'} \circ \gamma' + \gamma \circ h^A.
\end{align*}

Thus, by Remark \ref{rem: adjunction corresponds}, we conclude that $\gamma' = \overline{\gamma}$ in $H^0(\MA)$.
\end{proof}

\section{Abelian $n$-truncated DG-categories}\label{sec: abel}

\subsection{Additive $n$-truncated DG-categories}
Here, we explore the basic properties of the notion of $n$-tuncated DG-categories. Throughout this section, $\MA$ denotes a non-positive DG-category. 

\begin{definition}\label{def: n,1-dg}
Let $n$ be a positive integer. If the cohomology of $\MA(X, Y)$ vanishes in degrees below $-n$ for all $X, Y \in \MA$, then $\MA$ is called a \it{$n$-truncated DG-category}. More strictly, if $\MA$ satisfies $\MA(X,Y)^{\leq n}=0$, $\MA$ is called a \it{strict $n$-truncated DG-category}. A $\infty$-truncated DG-category is simply a non-positive DG-category.

If $\MA$ is additive (i.e.\ $H^0(\MA)$ is additive), it is called an \it{additive $n$-truncated DG-category}. 
\end{definition}

This terminology is consistent with $(n,1)$-categories in the sense of \cite[Definition 2.3.4.1]{lurie2009higher}) as shown in the following proposition.

\begin{proposition}
$\MA$ is a $n$-truncated DG-category if and only if $N_\dg(\MA)$ is categorical equivalent to an $(n,1)$-category.
\end{proposition}
\begin{proof}
For $X, Y \in \MA$, $i \leq 0$, and $x \in \mathrm{Map}_{N_\dg(\MA)}(X, Y)_0$, there exist bijections:
\[
\pi_i(\mathrm{Map}_{N_\dg(\MA)}(X, Y), x) \cong \pi_i(\mathrm{Map}_{N_\dg(\MA)}(X, Y), 0) \cong \pi_i(\mathrm{DK}(\MA(X, Y)), 0) \cong H^{-i}(\MA(X, Y)).
\]

The second isomorphism follows from \cite[Remark 1.3.1.12]{lurie2017higheralgebra}, and the third follows from the construction of the Dold-Kan correspondence. Here, $\mathrm{Map}_{N_\dg(\MA)(X, Y)}$ denotes the mapping space of $N_\dg(\MA)$ (cf.\ \cite[Definition 1.2.2.1]{lurie2009higher}), $\pi_i$ denotes the $i$-th homotopy group of a simplicial set, and $\mathrm{DK} \colon \RC(k)^{\leq 0} \to \mathrm{sSet}$ is the Dold-Kan correspondence. Thus, the vanishing of the cohomology of $\MA$'s morphism complexes below degree $-n$ is equivalent to the $n$-truncatedness of mapping spaces of $N_\dg(\MA)$. By \cite[Proposition 2.3.4.18]{lurie2009higher}, the claim follows.
\end{proof}

\begin{remark}
$\MA$ is additive if and only if $N_\dg(\MA)$ is additive as an $(\infty,1)$-category (cf.\ \cite[Definition 2.6]{gepner2016universality}). This equivalence follows from the fact that an $(\infty,1)$-category is additive if and only if its homotopy category is additive (cf.\ \cite[Proposition 2.8]{gepner2016universality}).
\end{remark}

\begin{example}
For a non-positive DG-category $\MC$, define the DG-category $\RD_\dg^n(\MC)$ as the full subcategory of $\tau_{\leq 0}(\RD_\dg(\MC))$ consisting of DG modules whose cohomologies are concentrated in degrees $-n+1$ through $0$. Then, $\RD_\dg^n(\MC)$ is an additive $n$-truncated DG-category.
\end{example}

\begin{remark}
If $\MC$ is a usual $k$-linear category, then $\RD_\dg^0(\MC)$ coincides with the category of right $\MC$-modules $\Mod \MC$.
\end{remark}

\begin{definition}\label{def: n-mono}
Let $f \colon X \to Y$ be a closed morphism of degree $0$. If the induced morphism in $\Mod H^0(\MA)$,
\[
H^i(f^\land) \colon H^i(X^\land) \to H^i(Y^\land),
\]
is a monomorphism for $i = -n+1$ and an isomorphism for $i \leq -n$, then $f$ is called an \it{$n$-monomorphism}. $f$ is called an \it{$n$-epimorphism} if it is $n$-monomorphism in the opposite DG-category $\MA^\op$. We regard any closed morphism as both an $\infty$-monomorphism and an $\infty$-epimorphism.

\end{definition}

\begin{remark}\label{rem: n-mono}
Concerning $n$-monomorphisms, note the following:
\begin{itemize}
    \item[(i)] The homotopy kernel of an $n$-monomorphism is generally not contractible.
    \item[(ii)] If $\MA$ is an $n$-truncated DG-category, then $f$ is an $n$-momomorphism if and only if $H^{-n+1}(f^\land)$ is a monomorphism in $\Mod H^0(\MA)$.
\end{itemize}
\end{remark}

When $\MA$ is a preabelian DG-category, its homotopy category naturally admits a pretriangulated structure. The notions of $n$-monomorphisms and $n$-epimorphisms can be characterized using the pretriangulated structure of the homotopy category as follows:

\begin{proposition}
Let $\MA$ be a finitely complete DG-category, and let $(\Omega, \Delta)$ denote the natural left triangulated structure on its homotopy category $H^0(\MA)$ (cf.\ Proposition \ref{prop: left triang}). For any closed morphism $f$ of degree $0$, the following are equivalent:
\begin{itemize}
    \item[(i)] $f$ is an $n$-monomorphism.
    \item[(ii)] $\Omega^{n-1}(f)$ is a monomorphism in $H^0(\MA)$, and $\Omega^{\geq n}(f)$ is an isomorphism in $H^0(\MA)$.
\end{itemize}
\end{proposition}
\begin{proof}
This follows directly from the fact that $\Omega X$ represents $H^{-1}(\MA(-,X))$ in $H^0(\MA)$.
\end{proof}

Dually, we have the following:

\begin{proposition}
Let $\MA$ be a finitely cocomplete DG-category, and let $(\Sigma, \nabla)$ denote the natural right triangulated structure on its homotopy category $H^0(\MA)$ (cf.\ Proposition \ref{prop: right triang}). For any degree 0 closed morphism $f$, the following are equivalent:
\begin{itemize}
    \item[(i)] $f$ is an $n$-epimorphism.
    \item[(ii)] $\Sigma^{n-1}(f)$ is an epimorphism in $H^0(\MA)$, and $\Sigma^{\geq n}(f)$ is an isomorphism in $H^0(\MA)$.
\end{itemize}
\end{proposition}

We now examine some basic properties of $n$-monomorphisms.

\begin{proposition}\label{proposition: homotopy kernel of n-mono}
Let $f \colon A \to A'$ be a closed morphism of degree $0$ such that the homotopy kernel $\Ker f$ exists. Then, $f$ is an $n$-monomorphism if and only if $H^i((\Ker f)^\land) = 0$ for all $i \leq -n+1$.
\end{proposition}
\begin{proof}
From Proposition \ref{prop: long exact induced by left exact}, we obtain the following long exact sequence:
\[
\begin{tikzcd}[row sep=0.3cm]
\cdots \ar[r, "\delta^i"] & H^i((\Ker f)^\land) \ar[r] & H^i(A^\land) \ar[r] & H^i(A'^\land) \\
\cdots \ar[r, "\delta^{-1}"] & H^{-1}((\Ker f)^\land) \ar[r] & H^{-1}(A^\land) \ar[r] & H^{-1}(A'^\land) \\
\ar[r, "\delta^0"] & H^0((\Ker f)^\land) \ar[r] & H^0(A^\land) \ar[r] & H^0(A'^\land).
\end{tikzcd}
\]

Since $f$ is an $n$-monomorphism, $H^i(f^\land)$ is an isomorphism for $i \leq -n$ and a monomorphism for $i = -n+1$, which implies $H^i((\Ker f)^\land) = 0$ for all $i \leq -n+1$. The converse follows by retracing the same argument in reverse.
\end{proof}

\begin{proposition}\label{lem: kernel morph is mono}
Let $\MA$ be an $n$-truncated DG-category, and let $(k_f, f, h_f)$ be a left exact $3$-term homotopy complex. Then, $k_f$ is an $n$-monomorphism.
\[
\begin{tikzcd}
\Ker f \ar[r, "k_f", swap] \ar[rr, "0", bend left=45, "{}" name=zero] & |[alias=Y]| A \ar[r, "f", swap] & A' \quad
\ar[Rightarrow, from=Y, to=zero, "h_f", shorten >=1mm]
\end{tikzcd}
\]
\end{proposition}
\begin{proof}
Since $\MA$ is an $n$-truncated DG-category, $H^i(A'^\land) = 0$ for all $i \leq -n+1$. From the long exact sequence in Proposition \ref{prop: long exact induced by left exact}, it follows that $k_f$ is an $n$-monomorphism.
\end{proof}

\subsection{Abelian $n$-truncated DG-categories and their homotopy categories}

\begin{definition}\label{def: abel def}
Let $\MA$ be an additive $n$-truncated DG-category. Then, $\MA$ is called an \it{abelian $n$-truncated DG-category} if it satisfies the following:
\begin{itemize}
    \item [(i)] $\MA$ is a preabelian $n$-truncated DG-category.
    \item [(ii)] For any right exact $3$-term homotopy complex:
    \[
    \begin{tikzcd}
    X \ar[r, "f", swap] \ar[rr, "0", bend left=45, "{}" name=zero]
    & |[alias=Y]| Y \ar[r, "g", swap] & Z
    \ar[Rightarrow, from=Y, to=zero, "h", shorten >=1mm]
    \end{tikzcd}
    \]
    if $f$ is an $n$-monomorphism, then this $3$-term homotopy complex is a homotopy short exact sequence.
    \item [(iii)] The dual of (ii).
\end{itemize}
\end{definition}

Fix an abelian $n$-truncated DG-category $\MA$. Since abelian $n$-truncated DG-categories are self-dual, dual statements can be derived by applying the same arguments to the opposite DG-category.

\begin{proposition}\label{prop: mono push}
$n$-monomorphisms are stable under homotopy pushouts. Specifically, in the following homotopy cocartesian diagram where $f$ is an $n$-monomorphism, $f'$ is also an $n$-monomorphism:
\begin{equation}\label{dia: square}
\begin{tikzcd}
X \ar[r, "f"] \ar[d, "a", swap] & |[alias=Z]| Y \ar[d, "b"] \\
|[alias=A]| X' \ar[r, "f'", swap] & Y'
\ar[Rightarrow, "s", from=Z, to=A, shorten <=2mm, shorten >=2mm]
\end{tikzcd} 
\end{equation}

\end{proposition}
\begin{proof}
Let $Z$ be an object of $\MA$ such that $Z \cong Y \oplus X'$ in $H^0(\MA)$. Then, we have the following closed morphisms of degree $0$:
\[
\begin{tikzcd}
Y \ar[r, transform canvas={yshift=3pt}, "i_Y"] & Z \ar[l, transform canvas={yshift=-3pt}, "p_Y"] \ar[r, transform canvas={yshift=3pt}, "p_{X'}"] & X' \ar[l, transform canvas={yshift=-3pt}, "i_{X'}"]
\end{tikzcd}
\]
and morphisms $u_Y \colon Y \to Y$, $u_{X'} \colon X' \to X'$, and $v \colon Z \to Z$ of degree $-1$ satisfying:
\begin{itemize}
    \item $d(u_Y) = p_Y \circ i_Y - \id_Y$,
    \item $d(u_{X'}) = p_{X'} \circ i_{X'} - \id_{X'}$,
    \item $d(v) = i_Y \circ p_Y + i_{X'} \circ p_{X'} - \id_Z$.
\end{itemize}

Using these, we define the following $3$-term homotopy complex $((f, a)^t, (-b, f'), s')$:
\begin{itemize}
    \item $(f, a)^t := i_Y \circ f + i_{X'} \circ a$,
    \item $(-b, f') := -b \circ p_Y + f' \circ p_{X'}$,
    \item $s' := s - b' \circ u_Y \circ f + f' \circ u_{X'} \circ a$.
\end{itemize}
$$
\begin{tikzcd}
X \ar[r, "(f{,}a)^t", swap]
\ar[rr, "0", bend left=45, "{}" name=zero] & |[alias=Y]| Z \ar[r, "(-b{,}f')", swap] & Y'
\ar[Rightarrow, from=Y, to=zero, "s'", shorten >=1mm]
\end{tikzcd}
$$

Here, $(f, a)^t$ and $(-b, f')$ agree with the morphisms $[f, a]^t$ and $[-b, f']$ induced by the direct sum in $H^0(\MA)$.

Since \eqref{dia: square} is homotopy cocartesian, the $3$-term homotopy complex is right exact. Furthermore, because $f$ is an $n$-monomorphism, $(f, a)^t$ is also an $n$-monomorphism by Remark \ref{rem: n-mono}(ii). By the definition of an abelian $n$-truncated DG-category, this $3$-term homotopy complex is a homotopy short exact sequence.

From Proposition \ref{prop: pullback of seq}, the homotopy kernels of $f$ and $f'$ are homotopy equivalent. By Proposition \ref{proposition: homotopy kernel of n-mono}, $f'$ is also an $n$-monomorphism.
\end{proof}

\begin{proposition}\label{prop: the homotopy category of abelian}
$\MA$ becomes an exact DG-category by considering every homotopy short exact sequence as a conflation. That is, $\MA$ forms an exact DG-category where inflations are $n$-monomorphisms and deflations are $n$-epimorphisms.
\end{proposition}
\begin{proof}
We verify the conditions for an exact DG-category (cf.\ Definition \ref{definition: exact dg}).

(Ex0): This follows immediately from Lemma \ref{lem: splits are homotopy short exact}.

(Ex1): It is clear that $n$-monomorphisms are closed under composition. The same holds for $n$-epimorphisms.

(Ex2)': By Lemma \ref{lem: characterization of finite completeness}, $\MA$ admits homotopy pullbacks. To show that $n$-monomorphisms are preserved under homotopy pushouts, it suffices to invoke Proposition \ref{prop: mono push}.

(Ex2): This is the dual of (Ex2)', and hence holds.
\end{proof}

From the above result and Theorem \ref{thm: enhance of et}, we obtain the following corollary.

\begin{cor}
$H^0(\MA)$ naturally admits the structure of an extriangulated category.
\end{cor}

Any conflation $A \xrightarrow{f} B \xrightarrow{g} C$ in $H^0(\MA)$ naturally extends to both a left triangle $\Omega C \to A \xrightarrow{f} B \xrightarrow{g} C$ and a right triangle $A \xrightarrow{f} B \xrightarrow{g} C \to \Sigma A$. Combining this with the natural long exact sequences in extriangulated categories, we obtain the following:

\begin{cor}\label{cor: long exact indused abelian}
Let $\MC = H^0(\MA)$, and let $(\MC, \mathbb{E}, \mathfrak{s})$ denote the natural extriangulated structure on $\MC$. For a conflation 
\begin{tikzcd}[column sep = 15]
{A}\ar[r,"f"] & {B}\ar[r,"g"] & {C}\ar[r,dashed,"\delta"] & {}
\end{tikzcd}
in this extriangulated category, the following long exact sequences arise:
\[
\begin{tikzcd}[row sep=0.3cm]
0 \ar[r] & \MC(-,\Omega^{n-1} A) \ar[r, "\Omega^{n-1} f \circ -"] & \MC(-,\Omega^{n-1} B) \ar[r, "\Omega^{n-1} g \circ -"] & \MC(-,\Omega^{n-1} C)
\ar[r] & \cdots \\
{} \ar[r] & \MC(-,\Omega A) \ar[r, "\Omega f \circ -"] & \MC(-,\Omega B) \ar[r, "\Omega g \circ -"] & \MC(-,\Omega C)
\ar[r] & {} \\
{} \ar[r] & \MC(-,A) \ar[r, "f \circ -"] & \MC(-,B) \ar[r, "g \circ -"] & \MC(-,C)
\ar[r, "\delta \cdot -"] & {} \\
{} \ar[r] & \mathbb{E}(-,A) \ar[r] & \mathbb{E}(-,B) \ar[r] & \mathbb{E}(-,C) \ar[r] & \cdots \\
\cdots \ar[r] & \mathbb{E}^m(-,A) \ar[r] & \mathbb{E}^m(-,B) \ar[r] & \mathbb{E}^m(-,C) \ar[r] & \cdots
\end{tikzcd}
\]

\[
\begin{tikzcd}[row sep=0.3cm]
0 \ar[r] & \MC(\Sigma^{n-1} C,-) \ar[r, "- \circ \Sigma^{n-1} g"] & \MC(\Sigma^{n-1} B,-) \ar[r, "- \circ \Sigma^{n-1} f"] & \MC(\Sigma^{n-1} A,-)
\ar[r] & \cdots \\
\cdots \ar[r] & \MC(\Sigma C,-) \ar[r, "- \circ \Sigma g"] & \MC(\Sigma B,-) \ar[r, "- \circ \Sigma f"] & \MC(\Sigma A,-)
\ar[r] & {} \\
{} \ar[r] & \MC(C,-) \ar[r, "- \circ g"] & \MC(B,-) \ar[r, "- \circ f"] & \MC(A,-)
\ar[r, "- \cdot \delta"] & {} \\
{} \ar[r] & \mathbb{E}(C,-) \ar[r] & \mathbb{E}(B,-) \ar[r] & \mathbb{E}(A,-) \ar[r] & \cdots \\
\cdots \ar[r] & \mathbb{E}^m(C,-) \ar[r] & \mathbb{E}^m(B,-) \ar[r] & \mathbb{E}^m(A,-) \ar[r] & \cdots
\end{tikzcd}
\]
\end{cor}
\begin{proof}
Since $\MA$ is an additive $n$-truncated DG-category, we have $\Omega^n = 0$ and $\Sigma^n = 0$. Considering that $H^0(\MA)$ is both an extriangulated category and a pretriangulated category, the claim follows from Proposition \ref{prop: long exact induced by left triangle} and Proposition \ref{prop: long exact induced by conflation}.
\end{proof}

\begin{cor}
Let $\MC = H^0(\MA)$. Then $(\MC, \mathbb{E}, \mathfrak{s})$ naturally possesses a negative extension $\mathbb{E}^{-1} = \MC(-, \Omega *) \cong \MC(\Sigma -, *)$ in the sense of \cite[Definition 2.3]{adachi2023intervals}.
\end{cor}

\begin{remark}
More generally, any algebraic extriangulated category possesses a negative extension $\mathbb{E}^{-1} := H^{-1}(\MA(-,*))$. In the present case, this negative extension is representable.
\end{remark}

\subsection{The factorization theorem}

In usual abelian categories, every morphism can be written as a composition of an epimorphism and a monomorphism. A similar result holds in abelian $n$-truncated DG-categories, which we now describe. Fix an abelian $n$-truncated DG-category $\MA$.

\begin{definition}
Let $f \colon A \to A'$ be a closed morphism of degree $0$. A diagram $(e, m, \eta)$ of the form
\[
\begin{tikzcd}
& |[alias=imf]| X \ar[rd, "m"] & \\
A \ar[ru, "e"] \ar[rr, "f" name=f, swap] & & A'
\ar[Rightarrow, "\eta", from=imf, to=f, shorten <=2mm, shorten >=2mm]
\end{tikzcd}
\]
is called a \it{factoriaztion} of $f$, where $|e| = |m| = 0$ and $|\eta| = -1$. If $m$ is an $n$-monomorphism and $e$ is a $1$-epimorphism, $(e, m, \eta)$ is called a \it{$[1,n]$-factorization} of $f$. Similarly, if $m$ is a $1$-monomorphism and $e$ is an $n$-epimorphism, $(e, m, \eta)$ is called a \it{$[n,1]$-factorization}.
\end{definition}

\begin{definition}
Let $(e, m, \eta)$ and $(e', m', \eta')$ be two factorizations of $f$. A $4$-tuple of morphisms $(t, h_e, h_m, \zeta)$ is called a \it{morphism of factorizations} from $(e', m', \eta')$ to $(e, m, \eta)$ if the following holds:
\[
\begin{tikzcd}
& |[alias=X_0]| X_0 \ar[rd, "m"] & \\
A \ar[ru, "e" name=m_0] \ar[rd, "e'", swap] & & A' \\
& |[alias=X_1]| X_1 \ar[ru, "m'" name=e_1, swap] \ar[uu, "t" description]
\ar[Rightarrow, from=X_1, to=m_0, shorten <=3mm, shorten >=3mm, "h_e" description, swap]
\ar[Rightarrow, from=X_0, to=e_1, shorten <=3mm, shorten >=3mm, "h_m" description, swap]
\end{tikzcd}
\]
Here, $\zeta \colon A \to A'$ is a degree $-2$ morphism satisfying $d(\zeta) = -h_m \circ e' - \eta' + m \circ h_e + \eta$. If $t$ is a homotopy equivalence, $(e, m, \eta)$ and $(e', m', \eta')$ are said to be \it{homotopy equivalent}.
\end{definition}

\begin{theorem}\label{thm: the factorization theorem}
Let $f$ be a closed morphism of degree $0$. Then, a $[1,n]$-factorization of $f$ exists and is unique up to homotopy equivalence.
\end{theorem}

\begin{proposition}\label{prop: 1,n-fac is}
Let $(e_1, m_n, \eta)$ be a $[1,n]$-factorization of $f$. Then $m_n$ is an $n$-monomorphism, and the homotopy cokernel of $m_n$ is homotopy equivalent to the homotopy cokernel of $f$ via a natural morphism.
\[
\begin{tikzcd}
 & |[alias=imf]| X \ar[rd, "m_n"] & \\
A \ar[ru, "e_1"] \ar[rr, "f" name=f, swap] & & A'
\ar[Rightarrow, "\eta", from=imf, to=f, shorten <=2mm, shorten >=2mm]
\end{tikzcd}
\]
\end{proposition}

\begin{proof}
The fact that $m_n$ is an $n$-monomorphism follows directly from the definition. We now prove the second claim.

By the octahedral axiom in $\RD(\MA)$, we obtain the following diagram where each row and column forms a distinguished triangle:
\[
\begin{tikzcd}
    \Cone m_n^\lor[-1] \ar[r] \ar[d, equal] & \Cone f^\lor[-1] \ar[r] \ar[d] & \Cone e^\lor[-1] \ar[d] \\
    \Cone m_n^\lor[-1] \ar[r] & {A'}^\lor \ar[r, "m_n^\lor"] \ar[d, "f^\lor"] & X^\lor \ar[d, "e^\lor"] \\
    & A^\lor \ar[r, equal] & A^\lor
\end{tikzcd}
\]
Since $e_1$ is a $1$-epimorphism, we have $H^{\leq 0}(\Cone e^\lor[-1]) = 0$. Using the long exact sequence associated with the top distinguished triangle in the above diagram, the following morphisms are isomorphisms:
\[
H^{\leq 0}(\Cone m_n^\lor[-1]) \to H^{\leq 0}(\Cone f^\lor[-1]).
\]
Thus, the homotopy cokernel of $m_n$ is homotopy equivalent to the homotopy cokernel of $f$.
\end{proof}

In fact, it is possible to construct factorizations satisfying the homotopy cokernel condition stated in the previous proposition.

\begin{lemma}\label{lem: exsistance of factorization}
Let $f$ be a closed morphism of degree $0$. Then, there exists a factorization $(e, m_n, \eta)$ of $f$ such that $m_n$ is an $n$-monomorphism and the homotopy cokernel of $m_n$ is homotopy equivalent to the homotopy cokernel of $f$.
\end{lemma}

\begin{proof}
We can construct the kernel of $f$ and its homotopy cokernel as follows:
\[
\begin{tikzcd}
\quad\,A\,\, \ar[r,"f",swap] \ar[rr,"0",bend left=45,"{}" name=zero]& |[alias=Y]|A' \ar[r,"c(f)",swap]& \Cok f
\ar[Rightarrow,from=Y,to=zero,"h^f",shorten >=1mm]
\end{tikzcd}
\quad
\begin{tikzcd}
\Ker(c(f)) \ar[r,"kc(f)",swap] \ar[rr,"0",bend left=45,"{}" name=zero]& |[alias=Y]|A' \ar[r,"c(f)",swap]& \Cok f\quad
\ar[Rightarrow,from=Y,to=zero,"h_{c(f)}",shorten >=1mm]
\end{tikzcd}
\]
By Proposition \ref{lem: kernel morph is mono}, $kc(f)$ is an $n$-monomorphism. Since $\Ker(c(f))$ is the homotopy kernel of $c(f)$, we can construct the following factorization of $f$:
\[
\begin{tikzcd}
|[alias=K]|\Ker(c(f)) \ar[r,"kc(f)"] \ar[rr,"0" name=zero,bend left=45,]& |[alias=Y]|A' \ar[r,"c(f)"]& \,\,\Cok f\quad \\
A \ar[u,"e",dashed] \ar[ru,"f" name=f,swap,bend right=15] \ar[rru,bend right=30,"0" name=zero',swap]& &
\ar[Rightarrow,to=zero,from=Y,"h_{c(f)}",shorten >=1mm]
\ar[Rightarrow,to=zero',from=Y,"h^f"swap,shorten <=1mm, shorten >=2mm]
\ar[Rightarrow,from=K,to=f,"\eta",shorten >=2mm, shorten <=2mm]
\end{tikzcd}
\]
By setting $m_n := kc(f)$, the factorization $(e, m_n, \eta)$ satisfies the desired properties.
\end{proof}

\noindent
Next, we aim to prove the converse of Proposition \ref{prop: 1,n-fac is}. As a preliminary step, we establish a weaker version of the converse.

\begin{lemma}\label{lem: e is n-1-epi}
Let $(e, m_n, \eta)$ be a factorization of $f$ such that $m_n$ is an $n$-monomorphism and the homotopy cokernel of $m_n$ is homotopy equivalent to the homotopy cokernel of $f$ via a natural morphism. Then $e$ is a $2$-epimorphism.
\[
\begin{tikzcd}
 & |[alias=imf]| X \ar[rd,"m_n"] & \\
A \ar[ru,"e"] \ar[rr,"f" name=f,swap] & & A'
\ar[Rightarrow,"\eta",from=imf,to=f,shorten <=2mm,shorten >=2mm]
\end{tikzcd}
\]
\end{lemma}

\begin{proof}
In $\RD(\MA^\op)$, we obtain the following diagram:
\[
\begin{tikzcd}
    \Cone m_n^\lor[-1] \ar[r] \ar[d,equal] & \Cone f^\lor[-1] \ar[r] \ar[d] & \Cone e^\lor[-1] \ar[d] \\
    \Cone m_n^\lor[-1] \ar[r] & {A'}^\lor \ar[r,"m_n^\lor"] \ar[d,"f^\lor"] & X^\lor \ar[d,"e^\lor"] \\
    & A^\lor \ar[r,equal] & A^\lor
\end{tikzcd}
\]
Here, the top row is a distinguished triangle. Consequently, we obtain the following long exact sequence in $\Mod H^0(\MA^\op)$:
\[
\begin{tikzcd}[row sep=0.3cm]
\cdots \ar[r,"\delta^i"] & H^i(\Cone m_n^\lor[-1]) \ar[r] & H^i(\Cone f^\lor[-1]) \ar[r] & H^i(\Cone e^\lor[-1]) \\
\cdots \ar[r,"\delta^{-1}"] & H^{-1}(\Cone m_n^\lor[-1]) \ar[r] & H^{-1}(\Cone f^\lor[-1]) \ar[r] & H^{-1}(\Cone e^\lor[-1]) \\
\ar[r,"\delta^{0}"] & H^{0}(\Cone m_n^\lor[-1]) \ar[r] & H^{0}(\Cone f^\lor[-1]) \ar[r] & H^{0}(\Cone e^\lor[-1]) 
\end{tikzcd}
\]

By the assumption, the morphism $H^i(\Cone m_n^\lor[-1]) \to H^i(\Cone f^\lor[-1])$ is an isomorphism for all $i \leq 0$. Hence, $H^i(\Cone e^\lor[-1]) = 0$ for all $i \leq -1$. By the dual of Proposition \ref{proposition: homotopy kernel of n-mono}, $e$ is a $2$-epimorphism.
\end{proof}

\begin{lemma}\label{lem: exists morphism of factorizations}
Let $(e,m_n,\eta)$ be a factorization of $f$ such that the homotopy cokernel of $m_n$ is homotopy equivalent to the homotopy cokernel of $f$ via a natural morphism, and let $(e',m'_n,\eta')$ be a factorization of $f$ where $m'_n$ is an $n$-monomorphism. Then, the following hold:
\begin{itemize}
    \item [(i)] There exists a morphism of factorizations from $(e,m_n,\eta)$ to $(e',m'_n,\eta')$.
    \item [(ii)] Let $(t,h_e,h_m,\zeta)$ and $(t',h'_e,h'_m,\zeta')$ be morphisms of factorizations from $(e,m_n,\eta)$ to $(e',m'_n,\eta')$. Then, $t$ and $t'$ are homotopic.
\end{itemize}
\end{lemma}

\begin{proof}
(i): By the dual of Lemma \ref{lem:mor extend to homotopy kernel}, we obtain the following diagram:
$$
\begin{tikzcd}
\,\,\,A \ar[r,"f"] \ar[d,"e'"] \ar[rr,bend left=45,"0" name=zero]
&|[alias=A']| A' \ar[r,"c(f)"] \ar[d,equal]
&|[alias=cokf]| \Cok f \ar[d,dashed,"g"] \\
|[alias=X']|\,\,\,X'\,\,\, \ar[r,"m'_n",swap] \ar[rr,bend right=45,"0" name=zero',swap]
&|[alias=A'']| A' \ar[r,"c(m'_n)",swap]
& \Cok m'_n
\ar[Rightarrow,to=X',from=A',"-\eta'",swap,shorten >=2mm,shorten <=2mm]
\ar[Rightarrow,to=A'',from=cokf,"\eta''",swap,shorten >=2mm,shorten <=2mm]
\ar[Rightarrow,to=zero,from=A',"h^f",swap,shorten >=1mm,shorten <=1mm]
\ar[Rightarrow,to=zero',from=A'',"h^{m'_n}",swap,shorten >=1mm,shorten <=1mm]
\end{tikzcd}
$$
Here, the top row is right exact and the bottom row is a homotopy short exact sequence. Note that there exists a morphism of 3-term homotopy complexes $(e',\id_{A'},g,-\eta',\eta'',\widetilde{\eta'})$, satisfying:
$$
d(\widetilde{\eta'}) = g \circ h^f - \eta'' \circ f + c(m'_n) \circ \eta' - h^{m'_n} \circ e'.
$$
Focusing on the right-hand square of the diagram, we obtain the following diagram and a morphism of 3-term homotopy complexes $(t,\id_{A'},g,-h_m,\eta'',-\widetilde{\eta''})$:
\begin{equation}\label{dia: oo}
\begin{tikzcd}
\,\,\,\,X\,\,\,\, \ar[r,"m_n"] \ar[d,"t",dashed] \ar[rr,bend left=45,"0" name=zero]
&|[alias=A']| A' \ar[r,"c(f)"] \ar[d,equal]
&|[alias=cokf]| \Cok f \ar[d,"g"] \\
|[alias=X']|\,\,\,\,X'\,\,\,\, \ar[r,"m'_n",swap] \ar[rr,bend right=45,"0" name=zero',swap]
&|[alias=A'']| A' \ar[r,"c(m'_n)",swap]
& \Cok m'_n
\ar[Rightarrow,to=X',from=A',"-h_m",swap,shorten >=2mm,shorten <=2mm]
\ar[Rightarrow,to=A'',from=cokf,"\eta''",swap,shorten >=2mm,shorten <=2mm]
\ar[Rightarrow,to=zero,from=A',"h_{c(f)}",swap,shorten >=1mm,shorten <=1mm]
\ar[Rightarrow,to=zero',from=A'',"h^{m'_n}",swap,shorten >=1mm,shorten <=1mm]
\end{tikzcd}
\end{equation}
Here, since the homotopy cokernel of $m_n$ is homotopy equivalent to that of $f$, we identify them. Additionally, note that the following holds:
$$
d(\widetilde{\eta''}) = h^{m'_n} \circ t - c(m'_n) \circ h_m + \eta'' \circ m_n - g \circ h_{c(f)}.
$$
By the universality of the homotopy kernel of $c(f)$, we obtain the following diagram and a morphism $p \colon A \to \Cok f$ of  degree $-2$, satisfying:
$
d(p) = h_{c(f)} \circ e - c(f) \circ \eta - h^f
$.
$$
\begin{tikzcd}
|[alias=K]|\,\,\,\,X\,\,\,\, \ar[r,"m_n"] \ar[rr,"0" name=zero,bend left =45,]
&|[alias=Y]| A' \ar[r,"c(f)"]
& \,\,\Cok f\\
\,\,\,\,A\,\,\,\, \ar[u,"e"] \ar[ru,"f" name=f,swap,bend right=15]
\ar[rru,bend right=30,"0" name=zero',swap] & &
\ar[Rightarrow,to=zero,from=Y,"h_{c(f)}",shorten >=1mm]
\ar[Rightarrow,to=zero',from=Y,"h^f"swap,shorten <=1mm, shorten >=2mm]
\ar[Rightarrow,from=K,to=f,"\eta",shorten >=2mm, shorten <=2mm]
\end{tikzcd}
$$
In this case, in the following diagram, $e$ has a filler $(\eta',-\widetilde{\eta'})$:
$$
\begin{tikzcd}
 &|[alias=A]| A \ar[ld,"e'",bend right=20]\ar[rr,"0"name=zero]\ar[rd,"f",swap]
 & &|[alias=cokf]| \Cok f\ar[rd,"g"] & \\
X' \ar[rr,"m'_n"{name=mn},swap]& &|[alias=A']| A' \ar[ru,"c(f)"]\ar[rr,"c(m'_n)"name=cmn',swap] & &\Cok m'_n
\ar[Rightarrow,"h^f", from=A',to=zero,shorten <=1mm,shorten >=1mm]
\ar[Rightarrow,"\eta''",from=cokf,to=cmn',shorten <=1mm,shorten >=1mm]
\end{tikzcd}
$$
In fact, $d(\eta') = f - m'_n \circ e'$ and $d(-\widetilde{\eta'}) = -g \circ h^f + \eta'' \circ f - c(m'_n) \circ \eta' + h^{m'_n} \circ e'$ hold. Similarly, $t \circ e$ also has a filler $(\eta + h_m \circ e, -\eta'' \circ \eta + g \circ p + \widetilde{\eta''} \circ e)$ in the following diagram:
$$
\begin{tikzcd}
 &|[alias=A]| A \ar[ld,"t \circ e",bend right=20]\ar[rr,"0"name=zero]\ar[rd,"f",swap]
 & &|[alias=cokf]| \Cok f \ar[rd,"g"] & \\
X' \ar[rr,"m'_n"name=mn,swap]& &|[alias=A']| A' \ar[ru,"c(f)"]\ar[rr,"c(m'_n)"name=cmn',swap] & &\Cok m'_n
\ar[Rightarrow,"h^f", from=A',to=zero,shorten <=1mm,shorten >=1mm]
\ar[Rightarrow,"\eta''",from=cokf,to=cmn',shorten <=1mm,shorten >=1mm]
\end{tikzcd}
$$
Thus, $e' - t \circ e$ has a filler in the following diagram:
$$
\begin{tikzcd}
\,\,\quad X'\,\,\, \ar[r,"m'_n"] \ar[rr,"0" name=zero,bend left =45,]
&|[alias=Y]| A' \ar[r,"c(m'_n)"] & \Cok m'_n\\
A \ar[ru,"0"swap]\ar[u,"e' - t \circ e"] \ar[rru,bend right=30,"0" name=zero',swap]& &
\ar[Rightarrow,from=Y,to=zero,"h^{m'_n}",shorten >=1mm]
\ar[Rightarrow,from=Y,to=zero',"0"swap,shorten <=1mm, shorten >=2mm]
\end{tikzcd}
$$
By Proposition~\ref{prop: torikae}, there exist a morphism $h_e \colon A \to X'$ of degree $-1$ and a morphism $\zeta \colon A \to A'$ of degree $-2$ such that $d(h_e) = e' - t \circ e$ and $d(\zeta) = \eta' - \eta - h_m \circ e + m'_n \circ h_e$. Therefore, $(t,h_e,h_m,\zeta)$ is a morphism of factorizations.

(ii): Let $(t',h'_e,h'_m,\zeta')$ be a morphism of factorizations from $(e,m_n,\eta)$ to $(e',m'_n,\eta')$. We need to show that $t'$ is homotopic to $t$ defined in the proof of (i). We obtain the following diagram. This diagram is obtained by replacing $t$ with $t'$ and $h_m$ with $h'_m$ in diagram \eqref{dia: oo}.
$$
\begin{tikzcd}
\,\,\,\,X\,\,\,\,\ar[r,"m_n"]\ar[d,"t'"]\ar[rr,bend left=45,"0" name=zero]
&|[alias=A']|A'\ar[r,"c(f)"]\ar[d,equal]
&|[alias=cokf]|\Cok f\ar[d,"g"]
\\
|[alias=X']|\,\,\,\,X'\,\,\,\,\ar[r,"m'_n"]\ar[rr,bend right=45,"0" name=zero',swap]
&|[alias=A'']|A'\ar[r,"c(m'_n)"]
&\Cok m'_n
\ar[Rightarrow,to=X',from=A',"-h'_m",swap,shorten >=2mm,shorten <=2mm]
\ar[Rightarrow,to=A'',from=cokf,"\eta''",swap,shorten >=2mm,shorten <=2mm]
\ar[Rightarrow,to=zero,from=A',"h_{c(f)}",swap,shorten >=1mm,shorten <=1mm]
\ar[Rightarrow,to=zero',from=A'',"h^{m'_n}",swap,shorten >=1mm,shorten <=1mm]
\end{tikzcd}
$$
We aim to show the existence of a morphism $\widetilde{\eta'''}$ of degree $-2$ such that:
$$
d(\widetilde{\eta'''}) = h^{m'_n} \circ t' - c(m'_n) \circ h'_m + \eta'' \circ m_n - g \circ h_{c(f)}.
$$
If this holds, then by the universality of the homotopy kernel, $t'$ is homotopic to $t$.
Here, the morphism $h^{m'_n} \circ t' - c(m'_n) \circ h'_m + \eta'' \circ m_n - g \circ h_{c(f)}$ is a closed morphism of degree $-1$. Indeed, we can calculate as follows:
\begin{align*}
d(h^{m'_n} \circ t' &- c(m'_n) \circ h'_m + \eta'' \circ m_n - g \circ h_{c(f)}) \\
=&\, d(h^{m'_n}) \circ t' - c(m'_n) \circ d(h'_m) + d(\eta'') \circ m_n - g \circ d(h_{c(f)}) \\
=&\,- c(m'_n) \circ m'_n \circ t' - c(m'_n) \circ (m_n - m'_n \circ t') + (c(m'_n) - g \circ c(f)) \circ m_n + g \circ c(f) \circ m_n \\
=&\, 0.
\end{align*}
Thus, by Lemma \ref{lem: e is n-1-epi}, it suffices to find a morphism $\xi \colon A \to \Cok m'_n$ of degree $-2$ such that:
$$
d(\xi) = (h^{m'_n} \circ t' - c(m'_n) \circ h'_m + \eta'' \circ m_n - g \circ h_{c(f)}) \circ e.
$$
Now, define $\xi$ as follows:
$$
\xi := -g \circ p - \widetilde{\eta'} + \eta'' \circ \eta + c(m'_n) \circ \zeta' + h^{m'_n} \circ h'_e.
$$
We verify that this $\xi$ satisfies the required condition:
\begin{align*}
d(-g \circ p &- \widetilde{\eta'} + \eta'' \circ \eta + c(m'_n) \circ \zeta' + h^{m'_n} \circ h'_e) \\
=& -g \circ d(p) - d(\widetilde{\eta'}) + d(\eta'') \circ \eta - \eta'' \circ d(\eta) + c(m'_n) \circ d(\zeta') + d(h^{m'_n}) \circ h'_e - h^{m'_n} \circ d(h'_e) \\
=& -g \circ (h_{c(f)} \circ e - c(f) \circ \eta - h^f) - (g \circ h^f - \eta'' \circ f + c(m'_n) \circ \eta' - h^{m'_n} \circ e') \\
 & + (c(m'_n) - g \circ c(f)) \circ \eta - \eta'' \circ (f - m \circ e) + c(m'_n) \circ (\eta' - \eta - h'_m \circ e + m'_n \circ h'_e) \\
 & - c(m'_n) \circ m'_n \circ h'_e - h^{m'_n} \circ (e' - t' \circ e) \\
=& -g \circ h_{c(f)} \circ e + \eta'' \circ m \circ e - c(m'_n) \circ h'_m \circ e + h^{m'_n} \circ t' \circ e.
\end{align*}
Thus, the proof is complete.
\end{proof}

\begin{proposition}\label{lem: left of factorization is epi}
Let $(e,m_n,\eta)$ be a factorization of $f \colon A \to A'$ such that $m_n$ is an $n$-monomorphism, and the homotopy cokernel of $m_n$ is homotopy equivalent to the homotopy cokernel of $f$ via a natural morphism. Then, $e$ is a 1-epimorphism.
\end{proposition}

\begin{proof}
By Lemma \ref{lem: e is n-1-epi}, it suffices to show that $H^0(e^\lor)$ is injective in $H^0(\MA^\op)$. Take a closed morphism $x \colon X \to Y$ of degree $0$ and a morphism $h \colon A \to Y$ of degree $-1$ such that $d(h) = -x \circ e$. We will show that $x$ is null-homotopic.

We take the homotopy kernel of $x$ as follows:
$$
\begin{tikzcd}
\Ker(x)\ar[r,"k(x)",swap]\ar[rr,"0",bend left =45,"{}"name=zero]&|[alias=Y]|X\ar[r,"x",swap]&\quad Y\quad
\ar[Rightarrow,from=Y,to=zero,"h_{x}",shorten >=1mm]
\end{tikzcd}
$$
By the universality of the homotopy kernel, we obtain the following diagram and a filler $(\overline{u},\widetilde{u})$ of $u$:
$$
\begin{tikzcd}
|[alias=K]|\Ker x\ar[r,"k(x)"]\ar[rr,"0" name=zero,bend left =45,]&|[alias=Y]|X\ar[r,"x"]&\quad Y\quad \\
A\ar[ru,"e" name=f,swap,bend right=15]\ar[rru,bend right=30,"0" name=zero',swap]\ar[u,"u",dashed]& &
\ar[Rightarrow,to=zero,from=Y,"h_x",shorten >=1mm]
\ar[Rightarrow,to=zero',from=Y,"h"swap,shorten <=1mm, shorten >=2mm,swap]
%\ar[Rightarrow,from=K,to=f,"h''",shorten >=2mm, shorten <=2mm]
\end{tikzcd}
$$
Thus, we now have two factorizations of $f$, $(e,m_n,\eta)$ and $(u,m_n \circ k(x),m_n \circ \overline{u} + \eta)$. By assumption, $m_n \circ k(x)$ is an $n$-monomorphism. By Lemma \ref{lem: exists morphism of factorizations}, there exists a morphism of factorizations $(t,h_e,h_m,\zeta)$ from $(e,m_n,\eta)$ to $(u,m_n \circ k(x),m_n \circ \overline{u} + \eta)$, which is unique up to homotopy.

$$
\begin{tikzcd}
&|[alias=X_0]|\Ker x\ar[rd,"m_n \circ k(x)"]& \\
A\ar[ru,"u"name=m_0]\ar[rd,"e",swap]& & A'\\
&|[alias=X_1]|X\ar[ru,"m_n"name=e_1,swap]\ar[uu,"t"description]
\ar[Rightarrow,from=X_1,to=m_0,shorten <=3mm,shorten >=3mm,"h_e"description,swap]
\ar[Rightarrow,from=X_0,to=e_1,shorten <=3mm,shorten >=3mm,"h_m"description,swap]
\end{tikzcd}
$$

Therefore, $(k(x) \circ t, k(x) \circ h_e + \overline{u}, h_m, \zeta)$ defines a morphism from $(e,m_n,\eta)$ to itself. By Lemma \ref{lem: exists morphism of factorizations} (ii), $k(x) \circ t - \id_X$ is null-homotopic. Since $k(x)$ is both an inflation and a retraction in $H^0(\MA)$, it is an isomorphism in $H^0(\MA)$. Thus, $x = 0$ in $H^0(\MA)$. This completes the proof.
\end{proof}

\begin{cor}\label{cor: uniqueness of factorization}
If $f$ admits two $[1, n]$-factorizations, then there exists a unique morphism between them, up to homotopy equivalence and it is a homotopy equivalence.
\end{cor}
\begin{proof}
Let $(e,m,\eta)$ and $(e',m',\eta')$ be two $[1,n]$-factorizations of $f$. By Proposition \ref{prop: 1,n-fac is}, the homotopy cokernels of $m$ and $m'$ are homotopy equivalent to the homotopy cokernel of $f$. Hence, by Lemma \ref{lem: exists morphism of factorizations}, there exists a morphism from $(e,m,\eta)$ to $(e',m',\eta')$. Furthermore, by Lemma \ref{lem: exists morphism of factorizations}, this morphism is uniquely determined up to homotopy. 

By applying the same argument in the reverse direction, we can confirm that this morphism is a homotopy equivalence.
% Let $(e, m, \eta)$ and $(e', m', \eta')$ be two $[1, n]$-factorizations of $f$. By Proposition \ref{prop: 1,n-fac is}, the homotopy cokernels of $m$ and $m'$ are both homotopy equivalent to the homotopy cokernel of $f$. Thus, by Lemma \ref{lem: exists morphism of factorizations}, there exists a unique morphism between $(e, m, \eta)$ and $(e', m', \eta')$, up to homotopy equivalence.
\end{proof}

Now, we are ready to prove Theorem \ref{thm: the factorization theorem}.

\begin{proof}[Proof for Theorem \ref{thm: the factorization theorem}]
By Lemma \ref{lem: exsistance of factorization}, there exists a factorization $(e, m_n, \eta)$ of $f$ such that $m_n$ is an $n$-monomorphism and its homotopy cokernel is naturally homotopy equivalent to the homotopy cokernel of $f$. Thus, by Proposition \ref{lem: left of factorization is epi}, $e$ is a $1$-epimorphism. The existence of a $[1, n]$-factorization follows, and its uniqueness is ensured by Corollary \ref{cor: uniqueness of factorization}.
\end{proof}

As the dual statement of Theorem \ref{thm: the factorization theorem}, the following also holds.

\begin{theorem}
Let $f$ be a degree-$0$ closed morphism. Then, the $[n, 1]$-factorization of $f$ exists and is unique up to homotopy equivalence.
\end{theorem}

In the language of the homotopy category, this theorem can be stated as follows:
\begin{cor}
For any morphism $f$ in the extriangulated category $H^0(\MA)$, there exists a unique factorization $f = m_n \circ e_1$, up to isomorphism, such that:
\begin{itemize}
    \item [(i)] $\Sigma^n(m_n)$ is a monomorphism in $H^0(\MA)$.
    \item [(ii)] $e_1$ is an epimorphism in $H^0(\MA)$, and $\Omega^{\geq 1}(e_1)$ is an isomorphism in $H^0(\MA)$.
\end{itemize}
\end{cor}

\begin{cor}
For any closed morphism $f$ of degree $0$, $f$ is an $n$-monomorphism (resp.\ a $1$-epimorphism) if and only if, in its $[1, n]$-factorization $(e_1, m_n, \eta)$, $e_1$ (resp.\ $m_n$) is a homotopy equivalence.
\end{cor}

\begin{proof}
This follows immediately from the uniqueness of the $[1, n]$-factorization.
\end{proof}

\subsection{Abelian $n$-truncated DG-categories induced by $t$-structures}

In this section, we introduce a method to construct abelian $n$-truncated DG-categories using $t$-structures. Throughout this section, $\MD_\dg$ is assumed to be a non-positive DG-category.

\begin{definition}[{\cite[Definition 6.1]{chen2024exact1}}]
Let $\MD_\dg$ be a preabelian DG-category. We say that $\MD_\dg$ is a \it{stable DG-category} if, for any $3$-term homotopy complex $(f, g, h)$ in $\MD_\dg$, the following are equivalent:
\begin{itemize}
    \item[(i)] $(f, g, h)$ forms a homotopy short exact sequence.
    \item[(ii)] $(f, g, h)$ is left exact.
    \item[(iii)] $(f, g, h)$ is right exact.
\end{itemize}
\end{definition}

\begin{remark}
This stable DG-category can be regarded as an abelian $(\infty,1)$-category because any closed morphism of degree $0$ is both an $\infty$-monomorphism and an $\infty$-epimorphism (cf.\ Definition \ref{def: n,1-dg}).
\end{remark}

\begin{example}[{\cite[Example 6.2]{chen2024exact1}}]
For a pretriangulated DG-category $\MA$, the category $\tauleq{0}\MA$ is a stable DG-category. Therefore, algebraic triangulated categories are always obtained from stable DG-categories.
\end{example}

\begin{remark}\label{rem: pretriang is triang}
IF $\MD_\dg$ is a stable DG-category,  the triangulated structure on the homotopy category $\MD$ coincides with the pretriangulated structure obtained in Theorem \ref{thm: pretriang}. Specifically, since left exactness and right exactness are equivalent, we have $\Delta = \nabla$, and it is straightforward to confirm that the loop functor and suspension functor are equivalence functors.
\end{remark}

The following results are known about stable DG-categories:

\begin{proposition}[{\cite[Proposition 6.4]{chen2024exact1}}]
Let $\MD_\dg$ be a stable DG-category. Then, by taking all homotopy short exact sequences as conflations, $\MD_\dg$ becomes an exact DG-category.
\end{proposition}

\begin{proposition}[{\cite[Theorem 6.5]{chen2024exact1}}]
Let $\MA$ be an exact DG-category. The extriangulated structure on the homotopy category $H^0(\MA)$ defines triangulated structure if and only if $\MA$ is an exact DG-category arising from a stable DG-category.
\end{proposition}

Hereafter, let $\MD$ be a stable DG-category and $\MD$ will always denote the homotopy category of a DG-category $\MD_\dg$.

\begin{definition}
A pair of full sub-DG-categories $(\MD^{\leq 0}_\dg, \MD^{\geq 0}_\dg)$ of $\MD_\dg$ is called a \it{$t$-structure} on $\MD_\dg$ if it forms a $t$-structure in the homotopy category $\MD$.
\end{definition}

\begin{remark}
The heart of the $t$-structure on $\MD$ is denoted by $\MD^0 := \MD^{\leq 0} \cap \MD^{\geq 0} \subset \MD$. The standard cohomological functor associated with this $t$-structure is written as $H^0_{\RD} \colon \MD \to \MD^0$.
\end{remark}

\begin{remark}
Let $(\MD^{\leq 0}_\dg, \MD^{\geq 0}_\dg)$ be a $t$-structure on $\MD_\dg$. For any $n \in \mathbb{Z}$, define the following full sub-DG-categories of $\MD_\dg$:
\[
\MD^{\leq n} := \{ X \in \MD_\dg \mid X \text{ belongs to } \MD^{\leq n} \text{ in } \MD \},
\]
\[
\MD^{\geq n} := \{ X \in \MD_\dg \mid X \text{ belongs to } \MD^{\geq n} \text{ in } \MD \}.
\]
\end{remark}

\begin{definition}
For $m \leq m'$, define the following full sub-DG-category of $\MD_\dg$:
\[
\MD_\dg^{[m, m']} := \MD_\dg^{\leq m'} \cap \MD_\dg^{\geq m}.
\]
The full DG-subcategory $\MD_\dg^{[-n+1, 0]}$ is called the \it{$n$-extended heart} of $(\MD^{\leq 0}_\dg, \MD^{\geq 0}_\dg)$.
\end{definition}

\begin{lemma}
For any $i \leq 0$ and $X,Y\in\MD_\dg$, the following holds:
\[
H^i(\MD_\dg(X, Y)) \cong \MD(X, Y[i]) \cong \MD(X[-i], Y).
\]
\end{lemma}

\begin{proof}
This follows immediately from Remark \ref{rem: pretriang is triang}.
\end{proof}

\begin{lemma}\label{lem: cohomology zero and object zero}
Let $(\MD_\dg^{\leq 0}, \MD_\dg^{\geq 0})$ be a $t$-structure on $\MD_\dg$. For any $m, n \in \mathbb{Z}$ and $X \in \MD_\dg^{[m, n]}$, the following are equivalent:
\begin{itemize}
    \item[(i)] $X$ is contractible in $\MD_\dg$, i.e., $X = 0$ in $\MD$.
    \item[(ii)] For all $m \leq i \leq n$, $H^i_\MD(X) = 0$ in $\MD$.
\end{itemize}
\end{lemma}

\begin{proof}
The implication from (i) to (ii) is clear. To show (ii) implies (i), note that for $X \in \MD^{[m, n]}$, we can write:
\[
X \in H^m_\MD(X)[-m] * H^{m+1}_\MD(X)[-m-1] * \cdots * H^n_\MD(X)[-n].
\]
By the assumption, $X = 0$ in $\MD$.
\end{proof}

\begin{lemma}\label{cohomology and representation commutative}
For any $X \in \MD^{[-n+1, 0]}$ and $0 \leq m \leq n-1$, the following are equivalent:
\begin{itemize}
    \item[(i)] $H^{\leq -m}_{\MD}(X) = 0$.
    \item[(ii)] $H^{\leq -m}(\MD_\dg(Z, X)) = 0$ holds for all $Z \in \MD^{[-n+1, 0]}_{\dg}$.
\end{itemize}
\end{lemma}

\begin{proof}
For any $Z \in \MD^{[-n+1, 0]}_{\dg}$ and $0 \leq m \leq l \leq n-1$, there is the following isomorphism:
\begin{align*}
H^{-l}(\MD^{[-n+1, 0]}_\dg(Z, X)) &\cong \MD(Z, X[-l]) \cong \MD(Z, \tauleq{0}(X[-l])).
\end{align*}
Thus, note that $H^{-l}(\MD_\dg(Z, X)) = 0$ for all $Z \in \MD_\dg^{[-n+1, 0]}$ if and only if $\tauleq{0}(X[-l]) = 0$.

(i) $\implies$ (ii): By assumption and Lemma \ref{lem: cohomology zero and object zero}, $\tauleq{0}(X[-l]) = 0$ holds in $\MD$ for all $m \leq l \leq n-1$. Therefore, from the above argument, $H^{-l}(\MD_\dg(Z, X)) = 0$ holds for all $m \leq l \leq n-1$ and $Z \in \MD_\dg^{[-n+1, 0]}$.

(ii) $\implies$ (i): Setting $l = m$ in the initial argument, we obtain that $\tauleq{0}(X[-m]) = 0$, which by Lemma \ref{lem: cohomology zero and object zero} is equivalent to $H^{\leq -m}_{\MD}(X) = 0$.
\end{proof}

\begin{theorem}\label{thm: induced by t-structure}
Let $\MD_\dg$ be a non-positive stable DG-category, and let $(\MD_\dg^{\leq 0}, \MD_\dg^{\geq 0})$ be its $t$-structure. Then, the $n$-extended heart $\MD_\dg^{[-n+1, 0]}$ is an abelian $(n, 1)$-category.
\end{theorem}

To prove this theorem, we examine $n$-monomorphisms, $n$-epimorphisms, homotopy kernels, and homotopy cokernels within the $n$-extended heart.

\begin{proposition}\label{prop: ker of t-str}
For any closed morphism $f \colon X \to Y$ of degree $0$ in $\MD^{[-n+1, 0]}_\dg$, the following diagram defines the homotopy kernel of $f$:
\[
\begin{tikzcd}
\tauleq{0}(\Cone f[-1]) \ar[r, "i"', swap] \ar[rr, "0", bend left=45, ""{name=zero}] & |[alias=Y]| X \ar[r, "f"', swap] & Y \quad\quad
\ar[Rightarrow, from=Y, to=zero, "h_f", shorten >=1mm]
\end{tikzcd}
\]
Here, $\tauleq{0}$ is the truncation functor associated with the $t$-structure on $\MD$.
\end{proposition}

\begin{proof}
For any $Z \in \MD^{[-n+1, 0]}$ and $l \leq 0$, the following isomorphisms hold:
\begin{align*}
H^l(\MD^{[-n+1, 0]}_\dg(Z, \tauleq{0}(\Cone f[-1]))) &\cong \MD(Z[-l], \tauleq{0}(\Cone f[-1])) \\
&\cong \MD(Z[-l], \Cone f[-1]) \\
&\cong H^l(\MD_\dg(Z, \Cone f[-1])) \\
&\cong H^l(\Cone f^\land_Z[-1]).
\end{align*}
The last isomorphism follows because $\Cone f[-1]$ defines the homotopy kernel of $f$ in $\MD_\dg$. Therefore, the above diagram defines the homotopy kernel of $f$.
\end{proof}

By considering the $t$-structure $((\MD^{\geq n-1})^\op, (\MD^{\leq n-1})^\op)$ on $\MD^\op$, this $n$-extended heart corresponds to the opposite DG-category of $\MD^{[-n+1, 0]}_\dg$. From the above discussion, we also obtain the following:

\begin{proposition}
For any closed morphism $f \colon X \to Y$ of degree $0$ in $\MD^{[-n+1, 0]}_\dg$, the following diagram defines the homotopy cokernel of $f$:
\[
\begin{tikzcd}
\quad\quad\quad X \ar[r, "f"', swap] \quad \ar[rr, "0", bend left=45, ""{name=zero}] & |[alias=Y]| Y \ar[r, "p"', swap] & \taugeq{-n+1}\Cone f
\ar[Rightarrow, from=Y, to=zero, "h^f", shorten >=1mm]
\end{tikzcd}
\]
\end{proposition}

\begin{proposition}\label{prop: n-mono in t-structure}
For any closed morphism $f \colon X \to Y$ of degree $0$ in $\MD_\dg^{[-n+1, 0]}$ and $0 \leq m \leq n-1$, the following are equivalent:
\begin{itemize}
    \item [(i)] $f$ is an $m$-monomorphism in $\MD_\dg^{[-n+1, 0]}$.
    \item [(ii)] $H^{-m+1}_{\MD}(f)$ is a monomorphism and $H^{\leq -m}_{\MD}(f)$ is an isomorphism in $\MD^0$.
\end{itemize}
\end{proposition}

\begin{proof}
By Lemma \ref{proposition: homotopy kernel of n-mono}, (i) is equivalent to the condition that for any $Z \in \MD_\dg^{[-n+1, 0]}$,
\[
H^{\leq -m}(\MD_\dg^{[-n+1, 0]}(Z, \tauleq{0}(\Cone f[-1]))) = 0.
\]
Moreover, by Lemma \ref{cohomology and representation commutative}, this is equivalent to $H^{\leq -m}_{\MD}(\tauleq{0}(\Cone f[-1])) = 0$. Therefore, (ii) holds.
\end{proof}

The dual statement is as follows:

\begin{proposition}
For any closed morphism $f \colon X \to Y$ of degree $0$ in $\MD_\dg^{[-n+1, 0]}$ and $0 \leq m \leq n-1$, the following are equivalent:
\begin{itemize}
    \item [(i)] $f$ is an $m$-epimorphism in $\MD_\dg^{[-n+1, 0]}$.
    \item [(ii)] $H^{-n+m}_{\MD}(f)$ is a epimorphism and $H^{\geq -n+m+1}_{\MD}(f)$ is an isomorphism in $\MD^0$.
\end{itemize}
\end{proposition}

\begin{proof}[Proof for Theorem \ref{thm: induced by t-structure}]
By Proposition \ref{prop: ker of t-str} and its dual, it is immediately clear that $\MD_\dg^{[-n+1, 0]}$ is a preabelian DG-category.

Next, consider the other conditions. Let $f \colon X \to Y$ be an $n$-monomorphism. The homotopy cokernel of $f$ is given as follows. We need to show that this right exact sequence is a homotopy short exact sequence in $\MD_\dg^{[-n+1, 0]}$:
\[
\begin{tikzcd}
\quad\quad X \quad \ar[r, "f"', swap] \ar[rr, "0", bend left=45, ""{name=zero}] & |[alias=Y]| Y \ar[r, "p"', swap] & \taugeq{-n+1}\Cone f
\ar[Rightarrow, from=Y, to=zero, "h^f", shorten >=1mm]
\end{tikzcd}
\]
Here, since $f$ is an $n$-monomorphism, it follows from Proposition \ref{prop: n-mono in t-structure} that $H^{-n}_{\MD}(\Cone f) = 0$. Hence, the morphism $\Cone f \to \taugeq{-n+1}\Cone f$ is an isomorphism in $\MD$. This implies that the above diagram forms a homotopy short exact sequence in $\MD_\dg$. Therefore, the sequence is also a short exact sequence in $\MD_\dg^{[-n+1, 0]}$.
\end{proof}

\begin{remark}
From the proof above, it is straightforward to verify the following equivalence for any $3$-term homotopy complex $(f, g, h)$ in $\MD_\dg^{[-n+1, 0]}$:
\begin{itemize}
    \item [(i)] $(f, g, h)$ is a homotopy short exact sequence in $\MD_\dg$.
    \item [(ii)] $(f, g, h)$ is a homotopy short exact sequence in $\MD_\dg^{[-n+1, 0]}$.
\end{itemize}
This implies that the exact DG-category structure obtained by considering $\MD_\dg^{[-n+1, 0]}$ as an extension-closed subcategory of $\MD_\dg$ corresponds to the exact DG-category structure obtained by treating $\MD_\dg^{[-n+1, 0]}$ as an abelian $n$-truncated DG-category.
\end{remark}

Without using the language of DG-categories, theorems so far are stated as follows:

\begin{theorem}\label{thm: h0}
Let $\MD$ be a algebraic triangulated categories and $(\MD^{\leq0},\MD^{\geq 0})$ be a $t$-structure. Then 
the following holds. 
\begin{itemize}
    \item[(i)] $\MD^{[-n+1,0]}$ admits a pretriangulated structure $(\Omega,\Sigma,\Delta,\nabla)$ where:
    \begin{itemize}
        \item $\Omega:=\tauleq{0}\circ[-1]$ and $\Sigma:=\taugeq{-n+1}\circ [1]$
        \item $\Delta$ is the class of triangles isomorphic to the following triangles:
        \[
        \begin{tikzcd}
            \tauleq{0}(Y[-1])\ar[r]&\tauleq{0}(\Cone f[-1]) \ar[r, swap]  & |[alias=Y]| X \ar[r, "f"', swap] & Y 
              \end{tikzcd}
              \]
            where $f\colon X\rightarrow Y$ is a morphism of $\MD^{[-n+1,0]}$.
             \item $\nabla$ is the class of triangles isomorphic to the following triangles:
             $$
             \begin{tikzcd}
              X\ar[r,"f"]&Y\ar[r]&\taugeq{-n+1}\Cone f\ar[r]&\taugeq{-n+1}(X[1])
        \end{tikzcd}$$
        \end{itemize}
    \item [(ii)] $\MD^{[-n+1,0]}$ admits a extriangulated structure suth that:
        \begin{itemize}
               \item Conflations are admissible triangles \begin{tikzcd}[column sep = 15]
{A}\ar[r,"f"] & {B}\ar[r,"g"] & {C}\ar[r,dashed,"\delta"] & {}
\end{tikzcd} in $\MD$ which every terms are in $\MD^{[-n+1,0]}$
                \item For a morphism $f\colon X\rightarrow Y$ in $\MD^{[-n+1,0]}$, The followings are equivalent:
                 \begin{itemize}
                      \item [(1)] $f$ is a inflation.
                        \item [(2)] $H_{\MD}^{-n+1}(f)$ is monomorphic in the heart of fixed $t$-structure.
                         \item [(3)] $\Omega^{n-1}(f)$ is monomorphic in $\MD^{[-n+1,0]}$.
                         \end{itemize}
                          \item For a morphism $f\colon X\rightarrow Y$ in $\MD^{[-n+1,0]}$, The followings are equivalent:
        \begin{itemize}
            \item [(1)] $f$ is a deflation.
            \item [(2)] $H_{\MD}^{0}(f)$ is epimorphic in the heart of fixed $t$-structure.
            \item [(3)] $\Sigma^{n-1}(f)$ is epimorphic in $\MD^{[-n+1,0]}$.
        \end{itemize}
         \end{itemize}
    \item [(iii)] For a morphism $f\colon X\rightarrow Y$ in $\MD^{[-n+1,0]}$, there exists a factorization $f=m_n\circ e_1$ suth that:
    \begin{itemize}
    \item $H^{-n+1}_\MD(m_n)$ is a monomorphism in the heart of fixed $t$-structure, or equivalently, $\Omega^{n-1}(m_n)$ is a monomorphism in $\MD^{[-n+1,0]}$.
        \item $H^{-n+1}_\MD (e_1)$ is a epimorphism and $H^{\geq -n+1}_\MD(e_1)$ is isomorphism in the heart of fixed $t$-structure, or equivalently, $e_1$ is a epimoprhism and $\Sigma^{\geq 1}(e_1)$ is a isomorphism in $\MD^{[-n+1,0]}$.
    \end{itemize}
    and this factorization is unique up to isomorphisms.
    \item [(iii)'] The dual statements of (iii)
    \end{itemize}
\end{theorem}

\section{Comparisons with other concepts}\label{sec :5}

\subsection{Relatively exact $2$-categories}

In this section, we explore the relationship between relatively exact $2$-categories introduced in \cite{nakaoka2008cohomology} and abelian $2$-truncated DG-categories. Specifically, we consider a weaker version of relatively exact $2$-categories, excluding condition $(\mathrm{a3}\text{-}1)$ in \cite{nakaoka2008cohomology}.

For the convenience of readers, we introduce the concept of relatively exact $2$-categories.

A symmetric categorical group serves as the $2$-dimensional analogue of an abelian group.

\begin{definition}
A \it{symmetric categorical group} $X$ is a symmetric monoidal category $(X, +, 0)$, where $0$ is the monoidal unit for the monoidal structure $+$, satisfying the following conditions:
\begin{itemize}
    \item Every morphism in $X$ is invertible.
    \item For every object $a \in X$, there exists an object $-a \in X$ and a morphism $\eta_a \colon a + (-a) \to 0$.
\end{itemize}
\end{definition}

\begin{remark}\label{rmk: connected components}
If $(X, +, 0)$ is a symmetric categorical group, then the set of connected components $\pi_0(X)$ of $X$ naturally forms an abelian group induced by $(+, 0)$.
\end{remark}

\begin{definition}
A $2$-category $\mathbf{S}$ is a $(2,1)$-category if every $2$-cell is invertible.
\end{definition}

\begin{definition}
For a $(2,1)$-category $\mathbf{S}$, we define the homotopy category $h(\mathbf{S})$ as follows:
\begin{itemize}
    \item The objects of $h(\mathbf{S})$ are the same as those of $\mathbf{S}$.
    \item For $X, Y \in \mathbf{S}$, $h(\mathbf{S})(X, Y) := \pi_0(\mathbf{S}(X, Y))$.
\end{itemize}
\end{definition}

\begin{definition}[{\cite[Definition 3.1]{nakaoka2008cohomology}}]
A $(2,1)$-category $\mathbf{S}$ is a \it{pre-locally-SCG-category} if the following conditions hold:
\begin{itemize}
    \item[(A1)] For any $X, Y \in \mathbf{S}_0$, the hom-category $\mathbf{S}(X, Y)$ admits a symmetric categorical structure $(+_{X, Y}, 0_{X, Y})$. And these structure satisfies $0_{Y, Z} \circ 0_{X, Y} = 0_{X, Z}$ for any $X,Y,Z\in\mathbf{Z}$.
    \item[(A2)] For any $1$-cell $f \colon X \to Y$, the induced functors $f \circ - \colon \mathbf{S}(Z, X) \to \mathbf{S}(Z, Y)$ and $- \circ f \colon \mathbf{S}(Y, Z) \to \mathbf{S}(X, Z)$ are monoidal with coherent map $(f\circ -)_I\colon f\circ 0_{Z,X}\rightarrow 0_{Z,Y}$ and $(-\circ f)_I\colon 0_{Y,Z}\circ f\rightarrow 0_{X,Z}$, satisfying:
    \begin{itemize}
        \item $(0_{Z, Y} \circ -)_I \colon 0_{Z, Y} \circ 0_{X, Y} \to 0_{X, Z}$ corresponds to $\mathrm{id}_{0_{X, Z}}$,
        \item $(- \circ 0_{X, Y})_I \colon 0_{Z, Y} \circ 0_{X, Y} \to 0_{X, Z}$ corresponds to $\mathrm{id}_{0_{X, Z}}$.
    \end{itemize}
\end{itemize}
\end{definition}

\begin{proposition}\label{prop: homotopy category of pre-locally-SGC}
Let $\mathbf{S}$ be a pre-locally-SCG-category. Then, $h(\mathbf{S})$ is pre-additive.
\end{proposition}
\begin{proof}
This follows directly from Remark \ref{rmk: connected components}.
\end{proof}

\begin{definition}[{\cite[Definition 3.1]{nakaoka2008cohomology}}]
\label{def: locally scg}
A pre-locally-SCG-category $\mathbf{S}$ is a \it{locally-SCG-category} if the following conditions hold:
\begin{itemize}
    \item[(A3)] There exists $0 \in \mathbf{S}_0$ such that:
    \begin{itemize}
        \item For any $X \in \mathbf{S}_0$ and $f \colon X \to 0$, $\mathbf{S}_2(f, 0_{X, 0})$ is a singleton.
        \item For any $X \in \mathbf{S}_0$ and $f \colon 0 \to X$, $\mathbf{S}_2(f, 0_{0, X})$ is a singleton.
    \end{itemize}
    \item[(A4)] $\mathbf{S}$ admits $2$-products (cf.\ \cite[Definition 2.5]{nakaoka2008cohomology}) and $2$-coproducts.
\end{itemize}
\end{definition}

\begin{remark}
The above definition is slightly different from \cite[Definition 3.1]{nakaoka2008cohomology}. \cite{nakaoka2008cohomology} requires the condition of (a3-1). 
\end{remark}

% We proof the diagram $X\xleftarrow{p_X} Z\xrightarrow{p_Y}Y$ defines a $2$-product of $X$ and $Y$. Take a diagram $X\xleftarrow{f_X} W\xrightarrow{f_Y}Y$ and define the $1$-cell $f:=i_X\circ f_X+i_Y\circ f_Y$ and $2$-cells $\epsilon_X:=u_X\cdot f_X$ and $\epsilon_Y:=u_Y\cdot f_Y$. Then we have the following diagram:
% $$
% \begin{tikzcd}
% X&|[alias=Z]|Z\ar[l,"p_X",swap]\ar[r,"p_Y"]& Y\\
%  & W\ar[lu, "f_X"name=fx,bend left=30]\ar[ru,"f_Y"name=fy,bend right=30,swap]\ar[u,"f"] &
%  \ar[Rightarrow,from=Z,to=fx,"\epsilon_X",shorten >=1mm]
%   \ar[Rightarrow,from=Z,to=fy,"\epsilon_Y",shorten >=1mm]
% \end{tikzcd}$$
% \textcolor{blue}{AAAAAAAAAAAAAAAAAAAAAAAAAAAAAAAAAAAAAAAAAAAAAAAAA}

\begin{definition}[{\cite[Definition 3.7]{nakaoka2008cohomology}}]
A locally-SCG-category $\mathbf{S}$ is a \it{relatively exact $2$-category} if the following conditions hold:
\begin{itemize}
    \item[(B1)] $\mathbf{S}$ admits $2$-kernels (cf.\ \cite[Definition 3.7]{nakaoka2008cohomology}) and $2$-cokernels.
    \item[(B2)] For any faithful $1$-cell (cf.\ \cite[Definition 2.6]{nakaoka2008cohomology}) $f$, $f$ defines the kernel of the cokernel of $f$.
    \item[(B3)] For any cofaithful $1$-cell $f$, $f$ defines the cokernel of the kernel of $f$.
\end{itemize}
\end{definition}

From now on, we fix a strict $2$-truncated DG-category $\MA$ (cf.\ Definition \ref{def: n,1-dg}).

\begin{definition}
We define a $2$-category $\MA_\SCG$ as follows:
\begin{itemize}
    \item $\mathrm{ob}(\MA_{\SCG}) := \mathrm{ob}(\MA)$,
    \item For any $X, Y \in \MA$, $\MA_{\SCG}(X, Y)_0 := \MA(X, Y)^0$,
    \item For any $f, g \in \MA(X, Y)^0$, $\MA_{\SCG}(f, g) := \{\eta \in \MA(X, Y)^{-1} \mid d(\eta) = g - f\}$,
    \item The composition of $1$-cells is the same as the composition in $\MA$,
    \item The vertical composition of $2$-cells is defined as their sum,
    \item The horizontal composition of $2$-cells $\phi \colon f \to f'$ and $\psi \colon g \to g'$ is defined as:
    \[
    \psi \bullet \phi := g \circ \phi + \psi \circ f' = g' \circ \phi + \psi \circ f.
    \]
\end{itemize}
\end{definition}

\begin{remark}
Since $\MA$ is a strict $2$-truncated DG-category, we have $\psi \circ \phi = 0$. Therefore:
\[
0 = d(\psi \circ \phi) = g' \circ \phi + \psi \circ f - g \circ \phi - \psi \circ f',
\]
which ensures that $\MA_\SCG$ satisfies the interchange law. Thus, $\MA_\SCG$ is a $2$-category. 
\end{remark}

\begin{proposition}
$H^0(\MA) = h(\MA_\SCG)$.
\end{proposition}
\begin{proof}
By definition, $H^0(\MA(X, Y)) = \pi_0(\MA_\SCG(X, Y))$.
\end{proof}

\begin{proposition}
$\MA_\SCG$ is a pre-locally-SCG-category.
\end{proposition}
\begin{proof}
We can verify that $(\MA_\SCG(X, Y), +, 0)$ forms a symmetric categorical group with coherent maps $\alpha = \mathrm{id}$, $\lambda = \mathrm{id}$, and $\gamma = \mathrm{id}$. Similarly, the condition (A2) can also be verified.
\end{proof}

\begin{lemma}\label{lem: additivity}
$\MA$ is additive if and only if $\MA_\SCG$ is a locally-SCG-category without condition (a3-1).
\end{lemma}
\begin{proof}
The "only if" direction is trivial. For the "if" direction, we proceed as follows:

We can easily verify that the zero object $0 \in H^0(\MA)$ satisfies condition (i) of Definition \ref{def: locally scg}. We now prove the existence of $2$-products in $\MA_\SCG$. Consider $X, Y \in \mathbf{S}_0$ and their direct sum $Z \cong X \oplus Y$ in $H^0(\MA)$. This direct sum is defined by the following diagram:
\[
\begin{tikzcd}
X \ar[r, transform canvas={yshift=3pt}, "i_X"] 
& Z \ar[l, transform canvas={yshift=-3pt}, "p_X"] 
  \ar[r, transform canvas={yshift=3pt}, "p_Y"] 
& Y \ar[l, transform canvas={yshift=-3pt}, "i_Y"]
\end{tikzcd}
\]
together with $2$-cells $u_X\colon p_X \circ i_X \to \id_X$, $u_Y\colon p_Y \circ i_Y \to \id_Y$, and $v\colon p_X \circ i_X + p_Y \circ i_Y \to \id_Z$. 

Next, consider a diagram $X \xleftarrow{f_X} W \xrightarrow{f_Y} Y$ in $\MA_\SCG$. The factorization condition (cf.\ Definition \cite[Definition 2.3]{nakaoka2008cohomology}) holds because $Z$ defines a product of $X$ and $Y$ in $H^0(\MA)$. We only check the uniqueness condition (cf.\ Definition \cite[Definition 2.3]{nakaoka2008cohomology}). 

Suppose there are two factorizations $(f, \epsilon_X, \epsilon_Y)$ and $(f', \epsilon'_X, \epsilon'_Y)$, and let $r\colon f \Rightarrow f'$ satisfy the following equations:
\[
(p_X \cdot r) \circ \epsilon'_X = \epsilon_X, \quad (p_Y \cdot r) \circ \epsilon'_Y = \epsilon_Y \quad \text{in } \MA_\SCG.
\]
In $\MA$, this implies:
\[
p_X \circ r + \epsilon'_X = \epsilon_X, \quad p_Y \circ r + \epsilon'_Y = \epsilon_Y.
\]
Now, we compute as follows:
\begin{align*}
r &= d(v) \circ r + i_X \circ (\epsilon_X - \epsilon'_X) + i_Y \circ (\epsilon_Y - \epsilon'_Y) \\
  &= v \circ d(r) + i_X \circ (\epsilon_X - \epsilon'_X) + i_Y \circ (\epsilon_Y - \epsilon'_Y) \\
  &= v \circ (f' - f) + i_X \circ (\epsilon_X - \epsilon'_X) + i_Y \circ (\epsilon_Y - \epsilon'_Y).
\end{align*}
Thus, the existence and uniqueness conditions are satisfied.
\end{proof}

\begin{proposition}\label{prop: 2-mono is faithful}
Let $f \colon X \to Y$ be a closed morphism of degree $0$. Then $f$ is a $2$-monomorphism if and only if $f$ is faithful in $\MA_\SCG$.
\end{proposition}
\begin{proof}
By \cite[Lemma 3.23]{nakaoka2008cohomology}, $f \colon X \to Y$ is faithful in $\MA_\SCG$ if and only if:
\[
f \circ - \colon \MA_\SCG(0_{A, X}, 0_{A, X}) \to \MA_\SCG(0_{A, Y}, 0_{A, Y})
\]
is injective. This corresponds to the following morphism:
\[
H^{-1}(f^\land) \colon H^{-1}(\MA(A, X)) \to H^{-1}(\MA(A, Y)).
\]
Thus, the statement holds.
\end{proof}

\begin{proposition}\label{prop: 2-kernel}
Let $f \colon A \to A'$ be a closed morphism of degree $0$. The following diagram:
\[
\begin{tikzcd}
K \ar[r, "k"] \ar[rr, "0", bend left = 45, "" name=zero] & |[alias=Y]| A \ar[r, "f"] & A'
\ar[Rightarrow, from=Y, to=zero, "h_f", shorten >=1mm]
\end{tikzcd}
\]
is a homotopy kernel of $f$ if and only if the following conditions hold:
\begin{itemize}
    \item[(i)] For any pair $(x, h_x)$, where $x \colon X \to A$ is a closed morphism of degree $0$ and $h_x \colon X \to A'$ is a morphism of degree $-1$ satisfying $d(h_x) = -f \circ x$, there exists a closed morphism $x' \colon X \to K$ of degree $0$ and morphism $\overline{x'} \colon X \to A$ of degree $-1$ such that $d(\overline{x'}) = x - k \circ x'$ and $h_f \circ x' = f \circ \overline{x'} + h_x$.
    \item[(ii)] If $(x'_0, \overline{x'_0})$ and $(x'_1, \overline{x'_1})$ satisfy (i), then there exists a unique morphism $r \colon X \to K$ of degree $-1$ such that $d(r) = x'_1 - x'_0$ and $k \circ r + \overline{x'_1} = \overline{x'_0}$.
\end{itemize}
\end{proposition}
\begin{proof}
Assume the above diagram is the homotopy kernel of $f$. By Proposition \ref{prop: univ of kernels}, we have a closed morphism $x'$ of degree $0$ and its filler $(\overline{x'},\widetilde{x'})$. Since $\mathcal{A}$ is an additive strict $(2,1)$-category, $\widetilde{x'}=0$ and 
\[
0=d(\widetilde{x'})=h_f\circ x'-f\circ \overline{x'}-h_x.
\]
Thus, (i) is verified. Next, let us verify (ii). Take $(x'_0,\overline{x'_0})$ and $(x'_1,\overline{x'_1})$ as in (ii). Then $x'_i$ has a filler $(\overline{x'_i},0)$, and by Proposition \ref{prop: torikae}, there exist a morphism $r\colon X\to K$ of degree $-1$ and a morphism $\overline{r}\colon X\to K$ of degree $-2$ such that $d(r)=x'_1-x'_0$ and $d(\overline{r})=(\overline{x'_1}-\overline{x'_0})+k\circ f$. However, $\overline{r}=0$ since $\overline{r}$ is a morphism of degree $-2$. Thus, we have $k\circ r+\overline{x'_1}=\overline{x'_0}$. To show the uniqueness of $r$, assume $r_i\colon X\to K$ satisfies $d(r_i)=x'_1-x'_0$ and $k\circ r_i+\overline{x'_1}=\overline{x'_0}$ for $i=1,2$. Then $r_1-r_0$ is a closed morphism of degree $-1$ and satisfies $k\circ (r_1-r_0)=0$. Thus, $r_1-r_0$ has a filler $(0,0)$. By Proposition \ref{prop: torikae}, there exists a morphism $r''\colon X\to K$ of degree $-2$ such that $d(r'')=r_1-r_0$. Since $r''$ is a morphism of degree $-2$, $r_1=r_0$ holds.

Conversely, assume (i) and (ii). It suffices to verify (i) and (ii) in Proposition \ref{prop: univ of kernels} for morphisms of degree  $0$ and $-1$. Take a pair of morphisms $(x,h_x)$ such that $x\colon X\to A$ is a morphism of degree $0$ and $h_x\colon X\rightarrow A$ satisfying $d(h_x)=-f\circ x$. Then, there exist $x'\colon X\to K$ and $\overline{x'}\colon X\to A$ such that $d(x')=0$ and $d(\overline{x'})=h_f\circ x'$. This implies $x'$ has a filler $(\overline{x'},0)$. Assume that $x'_0$ and $x'_1$ have fillers $(\overline{x'_0},0)$ and $(\overline{x'_1},0)$. By (ii) of this proposition, there exists $r\colon X\to K$ such that $d(r)=x'_1-x'_0$. Thus, the universality of the homotopy kernel in the case of degree $0$ holds. For a closed morphism $x\colon X\to A$ of degree $(-1)$-th such that $f\circ x=0$, pairs of morphisms $(0,0)$ and $(0,x)$ satisfy condition (i) in this proposition. Thus, by condition (ii), there exists $r\colon X\to K$ such that $d(r)=0$ and $k\circ r=x$. Hence, the universality of the homotopy kernel of $f$ is verified.
\end{proof}

\begin{proposition}\label{prop: left exact}
Consider the following diagram in $\MA$:
\[
\begin{tikzcd}
\quad K \ar[r,"k",swap] \ar[rr,"0",bend left =45,"{}"name=zero] & |[alias=Y]| A \ar[r,"f",swap] & A' \quad\,
\ar[Rightarrow,from=Y,to=zero,"h_f",shorten >=1mm]
\end{tikzcd}
\]
This diagram is left exact in $\MA$ if and only if it is left exact in $\MA_\SCG$.
\end{proposition}
\begin{proof}
By Proposition \ref{prop: 2-kernel}.
\end{proof}

% \begin{proof}
% By Proposition \ref{prop: 2-kernel}, the diagram is left exact in $\MA$ if and only if $k\colon K \to A$ satisfies the universal property of the homotopy kernel in $\MA$. Explicitly, for any object $X$ and any $3$-term homotopy  complex $(x, h_x)$, where $x\colon X \to A$ is a closed morphism and $h_x\colon X \to A'$ satisfies $d(h_x) = -f \circ x$, there exists a filler $(x', \overline{x'})$, and this filler is unique up to homotopy.

% In $\MA_\SCG$, left exactness is defined analogously: $k\colon K \to A$ must satisfy the universal property of a 2-kernel. By the construction of $\MA_\SCG$, the notions of homotopy kernels in $\MA$ and 2-kernels in $\MA_\SCG$ coincide.

% Thus, $k\colon K \to A$ satisfies the universal property of the homotopy kernel in $\MA$ if and only if it satisfies the universal property of the 2-kernel in $\MA_\SCG$. Therefore, the diagram is left exact in $\MA$ if and only if it is left exact in $\MA_\SCG$.
% \end{proof}

\begin{theorem}\label{thm: abelian 2-truncated}
Let $\MA$ be a $2$-truncated DG-category. Then $\MA$ is an abelian $(2,1)$-category if and only if $\MA_\SCG$ is a relatively exact $2$-category without (a3-1).
\end{theorem}
\begin{proof}
By Lemma \ref{lem: additivity}, $\MA$ is additive if and only if $\MA_\SCG$ is a pre-locally-SCG-category. The other conditions are equivalent by Proposition \ref{prop: left exact} and Proposition \ref{prop: 2-mono is faithful}.
\end{proof}

\subsection{Abelian $(n,1)$-categories}

In the paper \cite{stefanich2023derived}, abelian $(n,1)$-categories in the higher categorical settings are introduced. In this section, we use quasicategories as a model of $\infty$-categories. This subsection is based on \cite{chen2023exact}. Since the main claim essentially follows from a combination of these results, we will instead present the relationships here.

The main result of this subsection is the following. 

\begin{proposition}\label{thm: stef}
Let $\MA$ be a non-positive DG-categories. Then $\MA$ is an abelian $n$-truncated DG-category if and only if $N_{\dg}(\MA)$ is an abelian $(n,1)$-category in the sense of \cite{stefanich2023derived}.
\end{proposition}

\begin{remark}
In above theorem,  $N_{\dg}(\MA)$ is the DG-nerve of DG-category $\MA$. The definition and properties of DG-nerve functor $N_{\dg}\colon\dgcat\to\sset$ are written in \cite[1.3]{lurie2017higheralgebra}
\end{remark}

\begin{definition}[{\cite[Definition 6.2.4]{stefanich2023derived}}]
\label{def: ste}
Let $\MC$ be a quasicategory. $\MC$ is an \it{abelian $(n,1)$-category} if the following properties hold:
\begin{itemize}
    \item[(i)] $\MC$ an additive $(n,1)$-category with finite limits and colimits,
    \item[(ii)] For any $(n-2)$-truncated morphism $f\colon X\to Y$, a cofiber sequence $X\to Y\to \mathrm{cof}(f)$ is also a fiber sequence,
    \item[(iii)] The dual of (ii).
\end{itemize} 
\end{definition}

\begin{remark}
A morphism $f$ is $(n-2)$-truncated if $\mathrm{map}_{\MC}(C,X)\to\mathrm{map}_{\MC}(C,X)$ is $(n-2)$-truncated morphism between Kan complexes. Thus $f$ is n-monomorphism in DG-category $\MA$ if and only if $f$ is $(n-2)$-truncated in DG-nerve $N_{\dg}(\MA)$.
\end{remark}

To compare homotopy cartesian squares in DG-categories with pullbacks in $(\infty,1)$-categories, we impose a convenient assumption on DG-categories: the existence of pre-envelopes and pre-covers. It should be noted that any DG-category is quasi-equivalent to one satisfying these conditions. This is introduced by \cite[\S\S~5.25]{chen2023exact}.

\begin{definition}
Let $\MA$ be a non-positive DG-category. $\MA$ admits \it{pre-envelopes} (resp.\ \it{pre-covers}) if for any $X\in \MA$, there exist $\Cone(\id_X)\in \MA$ (resp.\ $\Cocone(\id_X)\in\MA$). 
\end{definition}

The following lemmas are easily valified. 

\begin{lemma}
Any additive non-positive DG-category is quasi-equivalent to a non-positive DG-category admits pre-envelopes and pre-covers such that $Z^0(\MA)$ is additive. 
\end{lemma}

\begin{lemma}
If $\MA$ is an abelian $n$-truncated DG-category and $\MA'$ is quasi-equivalent to $\MA$, $\MA'$ is also an abelian $n$-truncated DG categories.
\end{lemma}

Thus, we may assume that an abelian $n$-truncated DG-categories always admits pre-envelopes and pre-covers. 

\begin{lemma}\label{lem: torikae pb}
Let $\MA$ be a DG-category and  $X$ be the diagram below:
$$
\begin{tikzcd}
& X_1\ar[rd,"x_{12}"] & \\
X_0 \ar[ru,"x_{01}"]\ar[rr,"x_{02}"' name=x]& & X_2
\ar[from={1-2},to={x},Rightarrow, shorten >=1mm, shorten <=1mm,"h"]
\end{tikzcd}
$$
where $x_{01},x_{02},x_{12}$ are closed morphisms of degree $0$ and $h$ is a morphism of degree $-1$ and assume that there is $\Cone (\id_{X_0})$. Then $X$ is isomorphic to the strictly commutative diagram below in $\rep(\mathrm{Sq},\MA)$:
$$
\begin{tikzcd}
& \Cone(\id_{X_0})\oplus X_1\ar[rd,""] & \\
X_0 \ar[ru,"{\begin{bmatrix} \id_{X_0}\\ 0 \\x_{01}\end{bmatrix}}"]\ar[rr,"x_{02}"']& & X_2
\end{tikzcd}
$$
\end{lemma}
\begin{proof}
This is stated in \cite[\S ~ 5.25, p.~209]{chen2023exact}.
\end{proof}

\begin{lemma}[{\cite[Lemma 5.39]{chen2023exact}}]
\label{lem: pullback in A}
Let $\MA$ be a non-positive DG-category admitting pre-envelopes and pre-covers and $X$ be the following diagram in $\MA$:
$$
\begin{tikzcd}
X_{00}\ar[r]\ar[d]&X_{01}\ar[d]\\
X_{10}\ar[r]&X_{11}
\end{tikzcd}
$$
which is strictly commutative $\MA$. Then $X$ is homotopy pullback square in $\MA$ if and only if $X$ is pullback square in $N_{\dg}(\MA)$.
\end{lemma}

\begin{proof}[Proof for Proposition \ref{thm: stef}]
We can assume that $\MA$ admitting pre-envelopes and pre-covers and $Z^0(\MA)$ is additive. By definition of an additive $n$-truncated DG-category, we can easily verify that $\MA$ is additive $n$-truncated DG-category if and only if $N_{\dg}\MA$ is additive $(n,1)$-category. Thus, we show equivalenceness of (ii) in Definition \ref{def: abel def} and (ii) in Definition \ref{def: ste}.

Take an $n$-monomorphism $f\colon X\to Y$ and right exact sequence:
\[
\begin{tikzcd}
X \ar[r, "f", swap] \ar[rr, "0", bend left=45, "{}" name=zero]
& |[alias=Y]| Y \ar[r, "g", swap] & Z
\ar[Rightarrow, from=Y, to=zero, "h", shorten >=1mm]
\end{tikzcd}
\]
By Lemma \ref{lem: torikae pb}, we have a strictly commutative square $X'$ which is isomorphic to $X$ in $\rep (\mathrm{Sq},\MA)$:
$$
X:=
\begin{tikzcd}
X\ar[r]\ar[d]&Y\ar[d]\ar[ld,Rightarrow,shorten <=5, shorten >=5,"h"]\\
0\ar[r]&Z
\end{tikzcd}
,\quad
X':=
\begin{tikzcd}
X'\ar[r]\ar[d]&Y'\ar[d]\\
I\ar[r]&Z'
\end{tikzcd}
$$
Since these are isomorphic, we can conclude that
\begin{itemize}
    \item $f$ is isomorphic to $f'$ in $\mathrm{Arr}~ H^0(\MA)$, 
    \item $I$ is contractible in $\MA$,
    \item $X'$ is homotopy cocartesian square in $\MA$.
\end{itemize}
Therefore  $f'$ is $(n-2)$-truncated and $X'$ is a pushout square in $N_{\dg}(\MA)$ and thus, $X$ is a pullback square in $N_{\dg}(\MA)$. This implies $X'$ is homotopy cartesian in $\MA$. 

Conversely, we assume that (ii) in Definition \ref{def: ste} and take a $n$-monomorphism $f\colon X\to Y$ and diagram $X\in\mathrm{Fun}(\Delta^1\times \Delta^1,N_{\dg}(\MA))$ which is pushout. This is identified with the following:
$$
\begin{tikzcd}
X\ar[r,"f"]\ar[d]\ar[rd] & Y\ar[d]\ar[ld,Rightarrow,shorten >= 20]\\
I\ar[r]\ar[ru,Rightarrow,shorten >= 20] & Z
\end{tikzcd}
$$
Then we can replace this diagram with a strictly commutative diagram $X'$ which is isomorphic to $X$ in $\mathrm{Fun}(\Delta^1\times \Delta^1,N_{\dg}(\MA))$ by using Lemma \ref{lem: torikae pb}. By Lemma \ref{lem: pullback in A} and assumption, $X'$ is a pullback square in $N_{\dg}(\MA)$ and it implies $X$ is also a pullback square.
\end{proof}

\appendix

\section{Extriangulated categories}\label{app: A}

Extriangulated categories, introduced by \cite{nakaoka2019extriangulated}, provide a simultaneous generalization of exact categories and triangulated categories. In this section, we introduce the definition and discuss some basic properties.

Throughout this section, let $\MC$ be an additive $k$-linear category.

\begin{definition}
Let $\MC$ be an additive category, and let $A, C \in \MC$. Consider two sequences of morphisms in $\MC$, $A \to B \to C$ and $A \to B' \to C$. If there exists an isomorphism $B \xrightarrow{\cong} B'$ such that the following diagram commutes, then we say these two sequences are isomorphic and write $(A \to B \to C) \sim (A \to B' \to C)$:
\[
\begin{tikzcd}[row sep=0.3cm]
 & B \ar[rd] \ar[dd, "\cong"] & \\
A \ar[ru] \ar[rd] & & C \\
 & B' \ar[ru] & 
\end{tikzcd}
\]
The set of equivalence classes of such sequences is denoted by $\mathrm{Seq}(C, A)$. The image of $A \to B \to C$ in $\mathrm{Seq}(C, A)$ is written as $[A \to B \to C]$.

For $[A \to B \to C], [A' \to B' \to C'] \in \mathrm{Seq}(C, A)$, we define:
\[
[A \to B \to C] \oplus [A' \to B' \to C'] := [A \oplus A' \to B \oplus B' \to C \oplus C'].
\]
\end{definition}

\begin{definition}\cite{nakaoka2019extriangulated}
Let $\MC$ be an additive category, and let $\mathbb{E} \colon \MC^\op \times \MC \to \Mod k$ be a bilinear-functor. If a family $\mathfrak{s} = \{\mathfrak{s}_{C, A} \colon \mathrm{Seq}(C, A) \to \mathbb{E}(C, A)\}_{C, A \in \MC}$ satisfies the following conditions, it is called an \it{additive realization} of $\mathbb{E}$:
\begin{itemize}
    \item[(i)] For $\delta_0 \in \mathbb{E}(C_0, A_0)$, $\delta_1 \in \mathbb{E}(C_1, A_1)$, and a pair of morphisms $(a \colon A_0 \to A_1, c \colon C_0 \to C_1)$, suppose $\mathfrak{s}(\delta_i) = [A_i \to B_i \to C_i]$ for $i = 0, 1$. Then, there exists a morphism $b \colon B_0 \to B_1$ such that the following diagram commutes:
\[
\begin{tikzcd}
A_0 \ar[r] \ar[d, "a"] & B_0 \ar[r] \ar[d, "b", dashed] & C_0 \ar[d, "c"] \\
A_1 \ar[r] & B_1 \ar[r] & C_1
\end{tikzcd}
\]

\item[(ii)] $\mathfrak{s}(0_{C, A}) = [A \xrightarrow{[1, 0]^t} A \oplus C \xrightarrow{[0, 1]} C]$ holds.

\item[(iii)] For any $\delta_i \in \mathbb{E}(C_i, A_i)$ ($i = 0, 1$) and $\delta_0 \oplus \delta_1 \in \mathbb{E}(C_0 \oplus C_1, A_0 \oplus A_1)$, the following holds:
\[
\mathfrak{s}(\delta_0 \oplus \delta_1) = \mathfrak{s}(\delta_0) \oplus \mathfrak{s}(\delta_1),
\]
where $\delta_0 \oplus \delta_1 \in \mathbb{E}(C_0 \oplus C_1, A_0 \oplus A_1)$ corresponds to $(\delta_0, 0, 0, \delta_1)$ under the natural isomorphism:
\[
\mathbb{E}(C_0 \oplus C_1, A_0 \oplus A_1) \cong \mathbb{E}(C_0, A_0) \oplus \mathbb{E}(C_0, A_1) \oplus \mathbb{E}(C_1, A_0) \oplus \mathbb{E}(C_1, A_1).
\]
\end{itemize}
\end{definition}

\begin{remark} 
Note the following conventions regarding notation:
\begin{itemize}
    \item[(i)] When $\mathfrak{s}(\delta) = [A \to B \to C]$, we write:
    \[
    \begin{tikzcd}
    A \ar[r] & B \ar[r] & C \ar[r, "\delta", dashed] & {}
    \end{tikzcd}
    \]
    The above diagram is called a conflation, where $A \to B$ is called an inflation, and $B \to C$ is called a deflation.

    \item[(ii)] A triple $(a, b, c)$ of morphisms is called a morphism between conflations if $\mathbb{E}(C_0, a)(\delta_0) = \mathbb{E}(c, A_1)(\delta_1)$ and the following diagram commutes:
    \[
    \begin{tikzcd}
    A_0 \ar[r] \ar[d, "a"] & B_0 \ar[r] \ar[d, "b"] & C_0 \ar[r, dashed, "\delta_0"] & {} \\
    A_1 \ar[r] & B_1 \ar[r] & C_1 \ar[r, dashed, "\delta_1"] & {}
    \end{tikzcd}
    \]

    \item[(iii)] If $\mathbb{E}(C_0, a)(\delta_0) = \mathbb{E}(c, A_1)(\delta_1)$ holds, we write $a \cdot \delta_0 = \delta_1 \cdot c$.
\end{itemize}
\end{remark}

\begin{definition}\cite[Definition 2.12]{nakaoka2019extriangulated}
Let $\MC$ be an additive category, $\mathbb{E} \colon \MC^\op \times \MC \to \Mod k$ a bilinear-functor, and $\mathfrak{s}$ an additive realization of $\mathbb{E}$. The triple $(\MC, \mathbb{E}, \mathfrak{s})$ is called an extriangulated category if the following conditions are satisfied:
\begin{itemize}
    \item[(ET3)] For any conflations 
    \begin{tikzcd}[column sep = 15]
    {A}\ar[r] & {B}\ar[r] & {C}\ar[r,dashed,"\delta"] & {}
    \end{tikzcd} and 
\begin{tikzcd}[column sep = 15]
{A'}\ar[r] & {B'}\ar[r,"g"] & {C'}\ar[r,dashed,"\delta"] & {}
\end{tikzcd}
, and for any commutative diagram of the form:
    \[
    \begin{tikzcd}
    A \ar[r] \ar[d, "a"] & B \ar[r] \ar[d, "b"] & C \ar[r, dashed, "\delta'"] & {} \\
    A' \ar[r] & B' \ar[r] & C' \ar[r, dashed, "\delta'"] & {}
    \end{tikzcd}
    \]
    there exists a morphism $c \colon C \to C'$ such that $(a, b, c)$ is a morphism between conflations.

    \item[(ET3)'] The dual of (ET3).

    \item[(ET4)] For conflations $A \xrightarrow{f} B \xrightarrow{f'} D \dotarrow{\delta}$ and $B \xrightarrow{g} C \xrightarrow{g'} F \dotarrow{\delta'}$, there exists a following commutative diagram:
    \[
    \begin{tikzcd}
    A \ar[r, "f"] \ar[d, equal] & B \ar[r, "f'"] \ar[d, "g"] & D \ar[d, "d"] \ar[r, dashed, "\delta"] & {} \\
    A \ar[r, "g \circ f"] & C \ar[r] \ar[d, "g'"] & E \ar[d, "e"] \ar[r, dashed, "\delta''"] & {} \\
    & F \ar[r, equal] \ar[d, dashed, "\delta'"] & F \ar[d, dashed, "f' \cdot \delta'"] & {} \\
    & {} & {} & {}
    \end{tikzcd}
    \]
    where 
    \begin{itemize}
        \item [(i)] \begin{tikzcd}[column sep = 15]
    {D}\ar[r,"d"] & {E}\ar[r,"e"] & {F}\ar[r,dashed,"f'\cdot\delta'"] & {}
    \end{tikzcd} and
    \begin{tikzcd}[column sep = 15]
    {A}\ar[r,"g\circ f"] & {C}\ar[r,""] & {F}\ar[r,dashed,"\delta''"] & {}
    \end{tikzcd} are conflations,
        \item [(ii)] $\delta = \delta'' \cdot d$ and $f \cdot \delta'' = \delta' \cdot e$ hold.
    \end{itemize}
    \item[(ET4)'] The dual of (ET4).
\end{itemize}
\end{definition}

\begin{example}
We introduce typical examples (cf. \cite{nakaoka2019extriangulated}):
\begin{itemize}
    \item[(i)] Let $(\MC, \MS)$ be an exact category. In this case, an extriangulated category can be constructed with the bifunctor $\mathrm{Ext}^1_{\MC} \colon \MC^\op \times \MC \to \mathrm{Ab}$. Conversely, any extriangulated category in which inflations are monomorphisms and deflations are epimorphisms arises in this way (\cite[Corollary 3.18]{nakaoka2019extriangulated}).

    \item[(ii)] Let $\MT$ be a triangulated category. In this case, an extriangulated category can be constructed with the bifunctor $\MT(-, -[1]) \colon \MT^\op \times \MT \to \mathrm{Ab}$. Conversely, any extriangulated category in which every morphism is both an inflation and a deflation arises in this way (\cite[Proposition 3.22]{nakaoka2019extriangulated}).
\end{itemize}
\end{example}

\begin{definition}
Let $(\MC, \mathbb{E}, \mathfrak{s})$ be an extriangulated category, and let $\MC' \subset \MC$ be a subcategory closed under direct sums and isomorphisms. If every conflation 
\begin{tikzcd}[column sep = 15]
{A}\ar[r] & {B}\ar[r] & {C}\ar[r,dashed] & {}
\end{tikzcd}
satisfies $A, C \in \MC'$ implies $B \in \MC'$, then $\MC'$ is called an extension-closed subcategory.
\end{definition}

\begin{proposition}\cite[Remark 2.18]{nakaoka2019extriangulated}
Let $(\MC, \mathbb{E}, \mathfrak{s})$ be an extriangulated category. If $\MC' \subset \MC$ is extension-closed, then $(\MC', \mathbb{E}\mid_{\MC'^\op \times \MC'}, \mathfrak{s})$ is also an extriangulated category.
\end{proposition}

\begin{proposition}\cite[Theorem 3.5]{gorsky2021positive}\label{prop: long exact induced by conflation}
Let $(\MC, \mathbb{E}, \mathfrak{s})$ be an extriangulated category, and let 
\begin{tikzcd}[column sep = 15]
{A}\ar[r,"f"] & {B}\ar[r,"g"] & {C}\ar[r,dashed,"\delta"] & {}
\end{tikzcd}
be a conflation. Then, the following long exact sequence in $\Mod \MC$ is obtained:
\[
\begin{tikzcd}[row sep=0.3cm]
 & \MC(-, A) \ar[r, "f \circ -"] & \MC(-, B) \ar[r, "g \circ -"] & \MC(-, C)
\ar[r, "\delta \cdot -"] & {} \\
{}\ar[r] & \mathbb{E}(-, A) \ar[r] & \mathbb{E}(-, B) \ar[r] & \mathbb{E}(-, C) \ar[r] & \cdots \\
\cdots \ar[r] & \mathbb{E}^n(-, A) \ar[r] & \mathbb{E}^n(-, B) \ar[r] & \mathbb{E}^n(-, C) \ar[r] & \cdots \\
\end{tikzcd}
\]
Here, $\mathbb{E}^n \colon \MC^\op \times \MC \to \Mod k$ is a bifunctor called the positive extension (cf. \cite[Section 3]{gorsky2021positive}).
\end{proposition}

\section{Pretriangulated categories}\label{app: B}

In this section, we review the definition and basic properties of pretriangulated categories as introduced in \cite{beligiannis2007homological}. Note that the pretriangulated categories defined in \cite{beligiannis2007homological} are not equivalent to pretriangulated DG categories or triangulated categories without the condition (TR4).

\begin{definition}
For an additive category $\MC$ and an additive functor $\Omega \colon \MC \to \MC$, the category of triangles associated with $\Omega$, denoted $\mathrm{Tri}_\Omega$, is defined as follows:
\begin{itemize}
    \item Objects are sequences of morphisms in $\MC$ of the form $\Omega A \to C \to B \to A$.
    \item For two objects $\Omega A \to C \to B \to A$ and $\Omega A' \to C' \to B' \to A'$, a triple of morphisms $(\gamma, \beta, \alpha)$ is called a morphism between the triangles if the following diagram commutes:
    \[
    \begin{tikzcd}
    \Omega A \ar[r, "f"] \ar[d, "\Omega \alpha"] & C \ar[r, "g"] \ar[d, "\gamma"] & B \ar[r, "h"] \ar[d, "\beta"] & A \ar[d, "\alpha"] \\
    \Omega A' \ar[r, "f'"] & C' \ar[r, "g'"] & B' \ar[r, "h'"] & A'
    \end{tikzcd}
    \]
\end{itemize}
\end{definition}

\begin{definition}\cite[Definition 2.2]{beligiannis1994left}\label{def: left triangulated cat}
Let $\MC$ be an additive category, $\Omega \colon \MC \to \MC$ an additive functor, and $\Delta \subset \mathrm{Tri}_\Omega$ a full subcategory closed under isomorphisms. If the following conditions are satisfied, a triplet $(\mathcal{C}, \Omega, \Delta)$ is referred to as a \textit{left triangulated category}:
\begin{itemize}
    \item[(LT1)] 
    \begin{itemize}
        \item[(i)] $[0 \to A = A \to 0] \in \Delta$.
        \item[(ii)] For any morphism $f \in \MC(B, A)$, there exists $[\Omega A \to C \to B \xrightarrow{f} A] \in \Delta$.
    \end{itemize}
    \item[(LT2)] For $[\Omega A \xrightarrow{h} C \xrightarrow{g} B \xrightarrow{f} A] \in \Delta$, we have $[\Omega B \xrightarrow{-\Omega f} \Omega A \xrightarrow{h} C \xrightarrow{g} B] \in \Delta$.
    \item[(LT3)] For the following diagram, where the right-hand square commutes and each row belongs to $\Delta$:
    \[
    \begin{tikzcd}
    \Omega A \ar[r, "f"] \ar[d, "\Omega \alpha"] & C \ar[r, "g"] \ar[d, "\gamma", dashed] & B \ar[r, "h"] \ar[d, "\beta"] & A \ar[d, "\alpha"] \\
    \Omega A' \ar[r, "f'"] & C' \ar[r, "g'"] & B' \ar[r, "h'"] & A',
    \end{tikzcd}
    \]
    there exists a morphism $\gamma \colon C \to C'$ such that the diagram commutes.
    \item[(LT4)] For two triangles in $\Delta$, $\Omega A \to F \to B \xrightarrow{f} A$ and $\Omega B \to E \to C \xrightarrow{g} B$, there exists the following commutative diagram with $\Omega A \to D \to C \xrightarrow{f \circ g} A$ in $\Delta$:
    \[
    \begin{tikzcd}
    & \Omega B \ar[r, "\Omega f"] \ar[d] & \Omega A \ar[r, equal] \ar[d] & \Omega A \ar[d] \\
    \Omega F \ar[r] \ar[d] & E \ar[r] \ar[d, equal] & D \ar[r] \ar[d] & F \ar[d] \\
    \Omega B \ar[r] & E \ar[r] & C \ar[r, "g"] \ar[d, "f \circ g"] & B \ar[d, "f"] \\
    & & A \ar[r, equal] & A
    \end{tikzcd}
    \]
\end{itemize}
\end{definition}

The concept of a right triangulated category is defined dually.

\begin{definition}\cite[Chapter 2, Definition 1.1]{beligiannis2007homological}\label{def: pretr}
A tuple $(\MC, \Sigma, \Omega, \Delta, \nabla)$ is called a \it{pretriangulated category} if it satisfies the following conditions:
\begin{itemize}
    \item[(i)] $\Sigma \dashv \Omega$.
    \item[(ii)] $(\MC, \Omega, \Delta)$ is a left triangulated category, and $(\MC, \Sigma, \nabla)$ is a right triangulated category.
    \item[(iii)] For the following diagram in $\MC$, where the left-hand square commutes, the top row belongs to $\nabla$, and the bottom row belongs to $\Delta$:
    \[
    \begin{tikzcd}
    A \ar[r] \ar[d, "a"] & B \ar[r] \ar[d] & C \ar[d, dashed] \ar[r] & \Sigma A \ar[d, "\overline{a}"] \\
    \Omega C' \ar[r] & A' \ar[r] & B' \ar[r] & C'
    \end{tikzcd}
    \]
    there exists a morphism $C \to B'$ such that the diagram commutes. Here, $\overline{a} \colon \Sigma A \to C'$ is the morphism corresponding to $a \colon A \to \Omega C'$ under the adjunction.
    \item[(iv)] The dual of (iii).
\end{itemize}
\end{definition}

\begin{example} 
We present some typical examples:
\begin{itemize}
    \item[(i)] A triangulated category can always be regarded as a pretriangulated category with $\Delta = \nabla$ and $\Sigma \dashv \Omega$ as an equivalence adjoint. Conversely, a pretriangulated category satisfying these conditions is a triangulated category.
    \item[(ii)] An additive category $\MC$ with kernels and cokernels forms a pretriangulated category with $\Sigma = \Omega = 0$, where $\Delta$ consists of all left exact sequences and $\nabla$ consists of all right exact sequences.
\end{itemize}
\end{example}

\begin{remark}
Let $(\MC, \Omega, \Delta)$ be a left triangulated category, and let $\Omega C \to A \xrightarrow{f} B \xrightarrow{g} C$ be a triangle in $\Delta$. Then, the following sequence is exact:
\[
\begin{tikzcd}
\MC(-, A) \ar[r, "f \circ -"] & \MC(-, B) \ar[r, "g \circ -"] & \MC(-, C)
\end{tikzcd}
\]
\end{remark}

From the above remark and condition (LT2), we obtain the following long exact sequence:

\begin{proposition}\label{prop: long exact induced by left triangle}
Let $(\MC, \Omega, \Delta)$ be a left triangulated category, and let $\Omega C \to A \xrightarrow{f} B \xrightarrow{g} C$ be a triangle in $\Delta$. Then, the following long exact sequence is obtained:
\[
\begin{tikzcd}[row sep=0.3cm]
\cdots \ar[r] & \MC(-, \Omega^{n-1} A) \ar[r, "\Omega^{n-1} f \circ -"] & \MC(-, \Omega^{n-1} B) \ar[r, "\Omega^{n-1} g \circ -"] & \MC(-, \Omega^{n-1} C) \ar[r] & \cdots \\
{} \ar[r] & \MC(-, \Omega A) \ar[r, "\Omega f \circ -"] & \MC(-, \Omega B) \ar[r, "\Omega g \circ -"] & \MC(-, \Omega C) \ar[r] & {} \\
{} \ar[r] & \MC(-, A) \ar[r, "f \circ -"] & \MC(-, B) \ar[r, "g \circ -"] & \MC(-, C) & {}
\end{tikzcd}
\]
\end{proposition}

The corresponding statement for right triangulated categories is as follows:

\begin{proposition}\label{prop: long exact induced by right triangle}
Let $(\MC, \Sigma, \nabla)$ be a right triangulated category, and let $A \xrightarrow{f} B \xrightarrow{g} C \to \Sigma A$ be a triangle in $\nabla$. Then, the following long exact sequence is obtained:
\[
\begin{tikzcd}[row sep=0.3cm]
0 \ar[r] & \MC(\Sigma^{n-1} C, -) \ar[r, "- \circ \Sigma^{n-1} g"] & \MC(\Sigma^{n-1} B, -) \ar[r, "- \circ \Sigma^{n-1} f"] & \MC(\Sigma^{n-1} A, -) \ar[r] & \cdots \\
\cdots \ar[r] & \MC(\Sigma C, -) \ar[r, "- \circ \Sigma g"] & \MC(\Sigma B, -) \ar[r, "- \circ \Sigma f"] & \MC(\Sigma A, -) \ar[r] & {} \\
{} \ar[r] & \MC(C, -) \ar[r, "- \circ g"] & \MC(B, -) \ar[r, "- \circ f"] & \MC(A, -)
\end{tikzcd}
\]
\end{proposition}

\section*{Acknowledgements} 
The author gratefully acknowledges the invaluable guidance and insightful comments of his supervisor, Prof. Hiroyuki Nakaoka, which have greatly contributed to the improvement of this paper.

\bibliographystyle{mybstwithlabels}
\bibliography{myrefs}

\end{document}